\documentclass[10pt, a4paper]{amsart}
\usepackage[utf8]{inputenc}
\usepackage[T1]{fontenc}
\usepackage{tikz}
\usetikzlibrary{arrows.meta,patterns,positioning}
\usepackage{paralist}
\usetikzlibrary{arrows.meta}
\usetikzlibrary{calc}
\usepackage{amsmath}
\usepackage{amssymb}
\usepackage{amsthm}
\usepackage{graphicx}
\usepackage{url}
\usepackage{enumerate}
\usepackage{xcolor}
\usepackage{mathrsfs} 
\usepackage{mathtools}
\usepackage{t1enc}
\usepackage{tikz}
\usepackage{amsthm}
\usepackage{thmtools}
\usepackage{thm-restate}
\newtheorem{thm}{Theorem}

\newtheorem{lem}[thm]{Lemma}

\newtheorem{cor}[thm]{Corollary}

\theoremstyle{definition}
\newtheorem{defn}[thm]{Definition}
\newtheorem{rem}[thm]{Remark}

\DeclareMathOperator{\Pe}{\mathcal{P}} 
\DeclareMathOperator{\supp}{supp} 

\newcommand{\dbar}{\bar{d}}
\DeclareMathOperator{\dist}{dist}

\newcommand{\di}{D_P}

\DeclareMathOperator{\Qe}{\mathcal{Q}}

\DeclareMathOperator{\sint}{\overline{\kern-5pt\int}}
\DeclareMathOperator{\m}{\mathfrak m}

\newcommand{\Folner}{F{\o}lner }
\renewcommand{\phi}{\varphi}

\DeclareMathOperator{\M}{\mathcal{M}}

\DeclareMathOperator{\diam}{\text{diam}}

\DeclareMathOperator{\F}{\mathcal{F}}
\newcommand{\HD}{H}

\newcommand{\alf}{\mathscr{A}}

\def\ocirc#1{\ifmmode\setbox0=\hbox{$#1$}\dimen0=\ht0
    \advance\dimen0 by1pt\rlap{\hbox to\wd0{\hss\raise\dimen0
    \hbox{\hskip.2em$\scriptscriptstyle\circ$}\hss}}#1\else
    {\accent"17 #1}\fi}
\title[Entropy Density of Uniquely Ergodic Measures]{Entropy Density of Uniquely Ergodic Measures for Full Shifts over Amenable Residually Finite Groups}
\newcommand{\eps}{\varepsilon}
\newcommand{\R}{\mathbb{R}}
\newcommand{\Z}{\mathbb{Z}}
\newcommand{\N}{\mathbb{N}}

\author{Martha \L{}\k{a}cka}

\address[M. \L{}\k{a}cka]{
	Faculty of Mathematics and Computer Science, Jagiellonian University in Krak\'ow, ul. \L o\-jasiewicza 6, 30-348 Krak\'ow, Poland
}
\email{martha.ubik@uj.edu.pl}
\urladdr{www2.im.uj.edu.pl/MarthaLacka/}

\author{Marcel Mroczek}

\address[M. Mroczek]{
	Faculty of Mathematics and Computer Science, Jagiellonian University in Krak\'ow, ul. \L o\-jasiewicza 6, 30-348 Krak\'ow, Poland
}
\email{mroczek.marcel@gmail.com}

\begin{document}

\begin{abstract}
		We study entropy density for full shifts over amenable residually finite groups.  
			Entropy density means that every invariant measure can be approximated in the weak$^*$ topology, together with its entropy, by measures from a distinguished family.  
				In symbolic dynamics it is known, by a result of Weiss, that uniquely ergodic measures are entropy dense among ergodic measures for the shift action of $\mathbb Z$.  
						We extend this result to full shifts over amenable residually finite groups:
invariant measures supported on uniquely ergodic subsystems are entropy dense
in the collection of ergodic invariant measures.
							This applies in particular to full shifts over finitely generated abelian groups.
							The proof uses Cortez--Petite F{\o}lner tilings and a block-replacement
construction adapted to finite-index subgroups.

                        \end{abstract}



	\maketitle

\section*{Introduction}

Let $G$ be a countable amenable group and let $\alf$ be a finite alphabet.  
The full shift $\alf^G$ is one of the basic examples of a symbolic action of
$G$. Its invariant measures form a Choquet simplex, and a classical problem is
to understand which dynamically distinguished classes of measures are dense in
this simplex, both in the weak$^*$ topology and with respect to entropy.

In this paper we consider this question for measures supported on uniquely
ergodic subsystems. A family $\mathcal A\subset \M_G(\alf^G)$ is entropy dense
in a family $\mathcal B\subset \M_G(\alf^G)$ if every measure in
$\mathcal B$ can be approximated in the weak$^*$ topology by measures from
$\mathcal A$, with convergence of entropies at the same time. Thus entropy
density is stronger than weak$^*$ density: it requires the approximating measures
to preserve the entropy of the original measure in the limit.

For the shift action of $\mathbb Z$, Weiss proved that invariant measures
supported on uniquely ergodic subsystems are entropy dense in the set of ergodic
invariant measures~\cite{Weiss}. Our goal is to extend this result from
$\mathbb Z$-actions to full shifts over amenable residually finite groups.

A new difficulty in this setting is that an ergodic $G$-invariant measure need
not remain ergodic after restricting the action to a finite-index subgroup.
Moreover, the block-replacement maps used in the proof are not $G$-equivariant:
at a given level they are equivariant only with respect to the corresponding
finite-index subgroup. Thus one has to keep track of the associated finite
ergodic decompositions. We do this by introducing universally good generic
points and proving that their generic behavior is stable under the jigsaw
operations used in the construction.

Residual finiteness enters through the finite-index normal subgroups mentioned
above. We use a Cortez--Petite F{\o}lner sequence, whose fundamental domains
decompose into principal copies at smaller scales. This structure allows us to
perform block-replacement operations, called jigsaws, on principal copies while
keeping control of empirical distributions along the F{\o}lner sequence.

The main technical result says that every universally good generic point can be
approximated in the Besicovitch pseudometric by points whose orbit closures are
uniquely ergodic. Combining this approximation with continuity of entropy in the
$\bar d$-metric gives entropy density of invariant measures supported on
uniquely ergodic subsystems among ergodic invariant measures of the full shift.
In particular, the result applies to full shifts over finitely generated
abelian groups.
\section{Preliminaries}

\subsection{Basic notation}
For sets $A,B$ write $A\triangle B$ for the \emph{symmetric difference} and $|A|$ (or $\#A$) for \emph{cardinality}.  
Let $A^B$ be the set of all sequences indexed by $B$ with values in $A$, and
\[
A^{\infty}:=\{(a_i)_{i\in\N}: a_i\in A \text{ for all } i\in\N\}.
\]
A map $d\colon A\times A\to\R_+$ is a \emph{pseudometric} if it is symmetric and satisfies the triangle inequality.  
Pseudometrics $d_1,d_2$ on $A$ are \emph{uniformly equivalent} if for every $\eps>0$ there is $\delta>0$ with $d_i(x,y)<\delta\Rightarrow d_j(x,y)<\eps$ for $(i,j)\in\{(1,2),(2,1)\}$.  
Write $\mathcal C(X)$ for the space of continuous maps $X\to\R$.

\subsection{Measure-preserving systems}
Let $(X,\rho)$ be compact metric, $\diam(X)=1$, and $G$ countably infinite acting on $X$ by homeomorphisms.  
The pair $(X,G)$ is a \emph{topological dynamical system}.  
A \emph{measure-preserving system} is $(X,G,\mu)$ with $\mu$ a $G$-invariant Borel probability measure ($\mu(gA)=\mu(A)$ for all Borel $A\subset X$ and $g\in G$).

For a homeomorphism $g\colon X\to X$, $(X,g)$ denotes the $\Z$-system generated by $g$; if $g$ preserves $\mu$, write $(X,g,\mu)$.  
The \emph{product} of $\mathbf X=(X,G,\mu)$ and $\mathbf Y=(Y,G,\nu)$ is $\mathbf{X\times Y}=(X\times Y,G,\mu\times\nu)$ with $g(x,y)=(gx,gy)$.  
A system $\mathbf Y=(Y,G,\nu)$ is a \emph{factor} of $\mathbf X=(X,G,\mu)$ if there exists a full-measure, $G$-invariant $X'\subset X$ and a measurable, measure-preserving, $G$-equivariant map $X'\to Y$.

\subsection{Ergodicity}
A system $(X,G,\mu)$ is \emph{ergodic} if every measurable $G$-invariant function is $\mu$-a.e.\ constant.  
If the action is fixed, we also call such a measure $\mu$ \emph{ergodic}.  
See \cite[p.~233]{Ward} for equivalent characterizations.

\subsection{Amenable groups}\label{amen}
A sequence $(F_n)_{n\in\N}$ of finite non-empty subsets of $G$ is \emph{F{\o}lner} if
\[
\lim_{n\to\infty}\frac{|gF_n\triangle F_n|}{|F_n|}=0 \quad \text{for all } g\in G.
\]
Its elements are \emph{F{\o}lner sets}.  
A group is \emph{amenable} if it has a F{\o}lner sequence.

For finite $K,F\subset G$ and $\eps>0$, the set $F$ is \emph{$(K,\eps)$-invariant} if
\[
\frac{|KF\triangle F|}{|F|}<\eps, \quad KF=\{kf:k\in K,f\in F\}.
\]
If $(F_n)$ is F{\o}lner, then for every finite non-empty $K$ and $\eps>0$, all sufficiently large $n$ give $(K,\eps)$-invariance.

A F{\o}lner sequence $\F=(F_n)_{n\ge 1}$ is \emph{tempered} if there exists $C>0$ with
\[
\Big|\bigcup_{k\le n}F_k^{-1}F_{n+1}\Big|\le C\,|F_{n+1}| \quad \text{for all } n\in\N.
\]
Every F{\o}lner sequence has a tempered subsequence \cite[Prop.~1.5]{Lindenstrauss01}, hence every amenable group admits one.  
For tempered sequences we use:

\begin{thm}[{\cite[Theorem~3.3]{Lindenstrauss01}}]\label{pointwise}
	If $G$ acts by measure-preserving maps on $(X,\mathcal B,\mu)$ and $(F_n)$ is tempered, then for each $\phi\in L^1(\mu)$ there is a $G$-invariant $\bar\phi\in L^1(\mu)$ with
	\[
	\lim_{n\to\infty}\frac{1}{|F_n|}\sum_{f\in F_n}\phi(fx)=\bar\phi(x)\quad\text{for $\mu$-a.e.\ }x.
	\]
	If $(X,G,\mu)$ is ergodic, then the limit equals $\int_X\phi\,d\mu$ $\mu$-a.e.
\end{thm}

Amenable groups include all countable abelian groups.  
We assume throughout that $G$ is amenable.

\begin{lem}[{\cite[Lemma~2.6]{DHZ}}]\label{Z}
	Let finite non-empty $F,K\subset G$ and $\eps>0$. If $F$ is $(K,\eps/|K|)$-invariant, then
	\[
	|\{s\in F: Ks\subset F\}|\ge (1-\eps)|F|.
	\]
\end{lem}

\subsection{Residually finite groups}
A countable group $G$ is \emph{residually finite} if there exists a decreasing sequence $(H_n)_{n\in\N}$ of finite-index normal subgroups with $\bigcap_{n=0}^\infty H_n=\{e\}$ (the identity).  
For such $(H_n)$ and $g\in G$, write $[g]_m$ for the coset of $g$ in $G/H_m$.
Examples include free groups, finitely generated nilpotent and linear groups, and fundamental groups of $3$-manifolds.

\begin{thm}[{\cite[Lemma~4]{Cortez}}]\label{Cortez}
	If $G$ is amenable and residually finite, then there are finite-index normal subgroups $(H_n)_{n\ge0}$ and a F{\o}lner sequence $(F_n)_{n\in\N}$ such that:
	\begin{enumerate}
		\item $G=H_0\supset H_1\supset H_2\supset\cdots$ and $\bigcap_{n=0}^\infty H_n=\{e\}$,
		\item $\{e\}=F_0\subset F_1\subset F_2\subset\cdots$ and $\bigcup_{n=0}^\infty F_n=G$,
		\item for each $n$, $F_n$ is a fundamental domain of $G/H_n$,
		\item $F_{i+1}=\bigsqcup_{v\in F_{i+1}\cap H_i}F_iv$ for all $i\in\N$.
	\end{enumerate}
\end{thm}

\begin{rem}
	The F{\o}lner sequence in Theorem~\ref{Cortez} may be chosen tempered.
\end{rem}

\begin{defn}\label{Cortez-Petite}
	A F{\o}lner sequence with the properties of Theorem~\ref{Cortez} is a \emph{Cortez--Petite F{\o}lner sequence}; the chain $(H_n)$ is its \emph{associated Cortez--Petite sequence}.
\end{defn}

\subsection{Invariant measures and generic points}
Let $\M(X)$ be the set of Borel probability measures on $X$, equipped with the \emph{Prokhorov metric} $\di$:
\[
\di(\mu,\nu)=\inf\{\eps>0:\mu(B)\le \nu(B^\eps)+\eps\ \text{for all Borel }B\subset X\},
\]
where $B^\eps=\{y\in X:\exists x\in B \text{ with } \rho(x,y)<\eps\}$.  
Then $(\M(X),\di)$ is compact and $\di$ induces the weak$^*$ topology (see \cite[p.~238]{Billingsleybook}, \cite{Prohorov}).  
By the portmanteau theorem, $\di(\mu_n,\mu)\to0$ iff $\mu_n(A)\to\mu(A)$ for Borel $A$ with $\mu(\partial A)=0$.

Let $\M_G(X)$ be the simplex of $G$-invariant measures; for amenable $G$ it is non-empty.  
Let $\M_G^e(X)$ be its ergodic part.

For $z\in X$, let $\hat\delta_z$ be the Dirac measure at $z$.  
For finite non-empty $F\subset G$ and $x\in X$, define the \emph{empirical measure}
\[
\m(x,F)=\frac{1}{|F|}\sum_{f\in F}\hat\delta_{fx}.
\]
A point $x$ is \emph{generic} for $\mu\in\M_G(X)$ with respect to a F{\o}lner sequence $(F_n)$ if $\m(x,F_n)\to\mu$ in the weak$^*$ topology.  
If $\F$ is tempered and $\mu\in\M_G^e(X)$, then generic points for $\mu$ form a set of full $\mu$-measure (Thm.~\ref{pointwise}).  
A point is \emph{ergodic} if it is generic for some ergodic measure.

For $x\in X$ write $O_G(x)=\{gx:g\in G\}$ for its orbit.  
If $x$ is generic for $\mu$, then $\supp(\mu)\subset \overline{O_G(x)}$.  
A point $x$ is \emph{uniquely ergodic} if $\M_G(\overline{O_G(x)})$ is a singleton; then $x$ is generic for that measure.  
Write $\M_G(x)=\M_G(\overline{O_G(x)})$.

\subsection{Symbolic dynamics}
Let $\alf$ be a finite alphabet with the discrete topology.  
Equip $\alf^G$ with the product topology and the \emph{shift action} $(gx)_h=x_{hg}$.  
The set of finite \emph{words} is
\[
\alf^*=\{w\colon F\to\alf: F\subset G \text{ finite}\}.
\]
For finite non-empty $F\subset G$ and $x\in\alf^G$ define the \emph{cylinder}
\[
[x,F]=\{z\in\alf^G: z_{|F}=x_{|F}\}.
\]
For a word $w\in\alf^*$ with domain $F$, set $[w]=\{z\in\alf^G: z_{|F}=w\}$.  
Cylinders are clopen and form a countable base.

\subsection{Upper asymptotic density}
For finite non-empty $F\subset G$ and $A\subset G$, let $D_F(A)=|A\cap F|/|F|$.  
Given a F{\o}lner sequence $\F=(F_n)$, define the \emph{upper asymptotic density}
\[
\overline d_{\F}(A)=\limsup_{n\to\infty} D_{F_n}(A).
\]

\subsection{The Besicovitch pseudometric}
For a F{\o}lner sequence $\F=(F_n)$, define for $\underline x=(x_g)_{g\in G},\underline x'=(x'_g)_{g\in G}\in X^G$
\[
D_{B,\F}(\underline x,\underline x')=\limsup_{n\to\infty}\frac{1}{|F_n|}\sum_{g\in F_n}\rho(x_g,x'_g).
\]
For $x,x'\in X$ set
\[
D_{B,\F}(x,x')=D_{B,\F}\big((gx)_{g\in G},(gx')_{g\in G}\big).
\]

Let $\hat\omega_{\F}(x)$ be the set of weak$^*$ accumulation points of $(\m(x,F_n))$.  
It is non-empty, closed, and contained in $\M_G(X)$ (cf.\ \cite[Thm.~8.10]{Ward}).  
If $\mu\in\hat\omega_{\F}(x)$, we say that $x$ \emph{quasi-generates} $\mu$ along $\F$.

Let $\HD$ be the Hausdorff metric on non-empty compact subsets of $\M(X)$ induced by $\di$:
\[
\HD(A,B)=\max\Big\{\inf\{\eps>0:A\subset B^\eps\},\ \inf\{\eps>0:B\subset A^\eps\}\Big\}.
\]

We use the following (see \cite[Prop.~1, Cor.~1]{DI}, \cite[Cor.~11, Thm.~17]{Straszak}, \cite[Thm.~4, Thm.~7]{KLOg}).

\begin{thm}\label{HD}
	For every $\eps>0$ there exists $\delta>0$ such that $D_{B,\F}(\underline x,\underline x')<\delta$ implies $\HD(\hat\omega_{\F}(\underline x),\hat\omega_{\F}(\underline x'))<\eps$ for all $\underline x,\underline x'\in X^G$.  
	The choice of $\delta$ is independent of $\F$.
\end{thm}

\begin{thm}\label{unequ}
	Let $\alf$ be finite with the discrete topology.  
	For any F{\o}lner sequence $\F=(F_n)$, the Besicovitch pseudometric on $\alf^G$ (for any admissible metric) is uniformly equivalent to
	\[
	\bar d_{\F}(x,x')=\limsup_{n\to\infty}\frac{|\{f\in F_n: x_f\neq x'_f\}|}{|F_n|}.
	\]
	The modulus of uniform equivalence does not depend on $\F$.
\end{thm}

The topology induced by the Besicovitch pseudometric is not compact, so \emph{equivalence} and \emph{uniform equivalence} differ.

\subsection{Entropy}\label{entropy}

Let $G$ be an amenable group acting on $X$ by $\mu$-preserving measurable maps. For a finite measurable partition $\mathcal P$ and a finite set $F\subset G$ define
\[\mathcal P^F=\bigvee_{g\in F}g^{-1}\mathcal P\quad\text{where}\quad g^{-1}\mathcal P = \{g^{-1}A\, :\, A\in\mathcal P\}.\] The \emph{Shannon entropy} of $\mathcal P$ is given by
\[H_{\mu}(\mathcal P)= -\sum_{A\in\mathcal P}\mu(A)\log\mu(A).\]
The \emph{dynamical entropy of the action of $G$ with respect to a partition $ \mathcal P$ and the measure $\mu$} or the \emph{dynamical entropy of a measure with respect to a partition} (depending on what is fixed and what is our variable) is defined as \[h_{\mu}(\mathcal P, G) = \limsup\limits_{n\to\infty}\frac{1}{|F_n|}H_{\mu}(\mathcal P^{F_n}),\]
where $(F_n)_{n\in\N}$ is a \Folner sequence in $G$ (the above quantity does not depend on the choice of the \Folner sequence \cite{DFR, OrnsteinWeiss}).
The \emph{dynamical entropy} of the action of $G$ with respect to $\mu$ (or of $\mu$ with respect to the action) is obtained by taking the supremum of the above numbers over all finite measurable partitions. We will denote it by $h(G, \mu)$ or $h(\mu)$ if it will be clear from the context what action we have in mind.


The next theorem is folklore (see \cite{Bowen}), so for the sake of completeness we present its proof in Appendix \ref{appendix}.
\begin{restatable}{thm}{kalikow}\label{Kalikow}
	Let $\alf$ be a finite alphabet with $|\alf|\geq 2$, let $G$ be a countable amenable group, and
	let $\F=(F_n)_{n\in\N}$ be a F{\o}lner sequence in $G$. As usual, $G$ acts
	on $\alf^G$ by the shift. Consider two $G$-invariant measures $\mu$ and
	$\nu$ and ergodic points $x,z\in\alf^G$ generating $\mu$ and $\nu$,
	respectively. Then
	\[
		|h(\mu)-h(\nu)|
		\le
		h_2\bigl(\dbar_{\F}(\mu,\nu)\bigr)
		+
		\dbar_{\F}(\mu,\nu)\log(|\alf|-1).
	\]
	Here
	\[
		h_2(t)=-t\log t-(1-t)\log(1-t),
		\qquad h_2(0)=0,
	\]
	and $\dbar_{\F}(\mu,\nu)$ denotes the joining $\dbar$-distance between
	$\mu$ and $\nu$ described in Appendix~\ref{appendix}.
\end{restatable}

This corollary follows directly from Theorem~\ref{Kalikow} and \cite[Theorem 3.13]{BLM}.
\begin{cor}\label{cor:entropy-continuity}
	Entropy is uniformly continuous with respect to $\dbar$ on the set of points generating ergodic measures. 
\end{cor}

\section{Entropy density for full shifts over residually finite amenable groups}\label{geste}

We say that a subset $A\subset \M_G(X)$ is \emph{entropy dense} in
$B\subset \M_G(X)$ if for every $\mu\in B$ there exists a sequence
$(\mu_n)_{n\in\N}\subset A$ such that
\[
  \mu_n \to \mu \quad\text{in the weak$^*$ topology}
  \quad\text{and}\quad
  h(\mu_n)\to h(\mu).
\]
The goal of this section is to prove Corollary~\ref{glowny}.

\subsection{Auxiliary Lemmas}

We begin with a series of lemmas. For notational convenience, we introduce Definition~\ref{epsequal}.
\begin{defn}\label{epsequal}
	For $\eps>0$ we say that $s,t\in\R$ are \emph{$\eps$-close} if $|s-t|<\eps$.
	\end{defn}

\begin{lem}\label{lemma:2}
	Fix $x\in\alf^G$ and a tempered \Folner sequence $\F=(F_j)_{j\in\N}$ in $G$. If for every $w\in\alf^*$ and $\eps>0$ there is $K\in\N$ for which all numbers in the set
	\[\big\{\m(x,F_kg)([w])\,:\,g\in G,\,k\geq K\big\}\] are pairwise $\eps$-close, then $x$ is uniquely ergodic.
\end{lem}

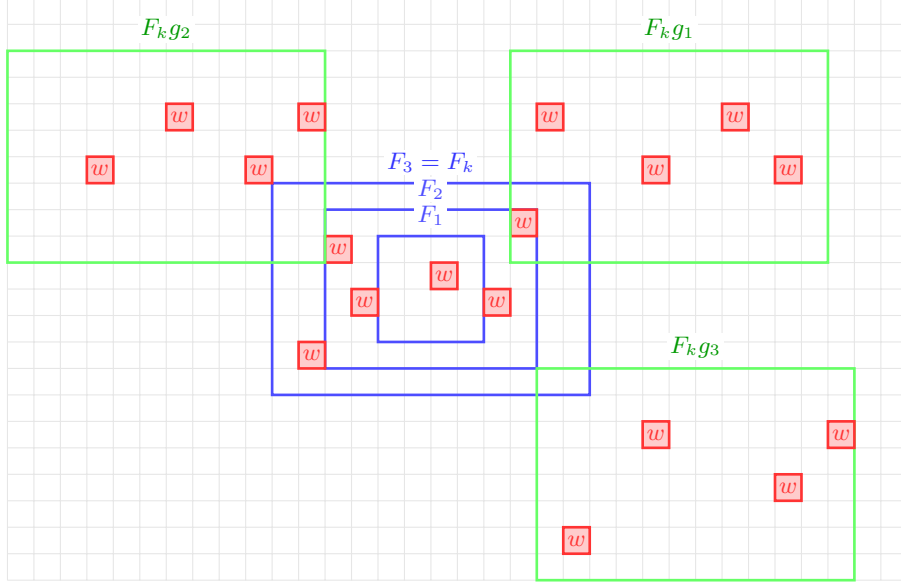
\begin{figure}[h]
	\centering
	\begin{tikzpicture}[font=\small, scale=0.7]
		
		\tikzset{
			folner/.style={draw=blue!70, line width=1pt},
			shiftset/.style={draw=green!60, line width=1pt},
			wtile/.style={fill=red!20, draw=red!80, line width=1pt}
		}
		
		\draw[step=0.5cm, gray!20, thin] (-8,-5.5) grid (9,5.5);
		
		\draw[folner] (-1,-1) rectangle (1,1);      
		\draw[folner] (-2,-1.5) rectangle (2,1.5);  
		\draw[folner] (-3,-2) rectangle (3,2);      
		
		\node[text=blue!70, fill=white, inner sep=1pt, anchor=south] at (0,1.15) {$F_1$};
		\node[text=blue!70, fill=white, inner sep=1pt, anchor=south] at (0,1.65) {$F_2$};
		\node[text=blue!80, fill=white, inner sep=1pt, anchor=south] at (0,2.15) {$F_3 = F_k$};
		
		\newcommand{\wtileLL}[2]{%
			\draw[wtile] (#1,#2) rectangle (#1+0.5,#2+0.5);
			\node[red!80] at (#1+0.25,#2+0.25) {$w$};
		}
		
		\wtileLL{-2.5}{-1.5}
		\wtileLL{-1.5}{-0.5}
		\wtileLL{0.0}{0.0}
		\wtileLL{1.0}{-0.5}
		\wtileLL{1.5}{1.0}
		\wtileLL{-2.0}{0.5}
		
		\begin{scope}[shift={(4.5,2.5)}]
			\draw[shiftset] (-3,-2) rectangle (3,2);
			\node[text=green!60!black, fill=white, inner sep=1pt, anchor=south] at (0,2.15) {$F_k g_1$};
			\wtileLL{-2.5}{0.5}
			\wtileLL{-0.5}{-0.5}
			\wtileLL{1.0}{0.5}
			\wtileLL{2.0}{-0.5}
		\end{scope}
		
		\begin{scope}[shift={(-5,2.5)}]
			\draw[shiftset] (-3,-2) rectangle (3,2);
			\node[text=green!60!black, fill=white, inner sep=1pt, anchor=south] at (0,2.15) {$F_k g_2$};
			\wtileLL{-1.5}{-0.5}
			\wtileLL{0.0}{0.5}
			\wtileLL{1.5}{-0.5}
			\wtileLL{2.5}{0.5}
		\end{scope}
		
		\begin{scope}[shift={(5,-3.5)}]
			\draw[shiftset] (-3,-2) rectangle (3,2);
			\node[text=green!60!black, fill=white, inner sep=1pt, anchor=south] at (0,2.15) {$F_k g_3$};
			\wtileLL{-2.5}{-1.5}
			\wtileLL{-1.0}{0.5}
			\wtileLL{1.5}{-0.5}
			\wtileLL{2.5}{0.5}
		\end{scope}
		
	\end{tikzpicture}
\caption{In the above example we have \(|F_k|=96\), 
	 \(\m(x,F_k)([w])=1/16\) and \(\m(x,F_k g_i)([w])=1/24\) for each
	 \(i=1,2,3\). These numbers are $\eps$-close for every $\eps>1/48$.}
	
\end{figure}

\begin{proof}
	Fix $\mu\in \M_G(x)$. We use the construction from the proof
	of~\cite[Lemma~22]{Straszak}. It gives the following approximation scheme:
	for every $\eta>0$ there are $p=p(\eta)\in\N$ and tuples
	\[
		\bigl(g_1^{(j,\eta)},\dots,g_p^{(j,\eta)}\bigr)\in G^p,
		\qquad j\in\N,
	\]
	such that
	\[
		F_j g_{i_1}^{(j,\eta)}
		\cap
		F_j g_{i_2}^{(j,\eta)}
		=
		\emptyset
		\quad\text{whenever } i_1\neq i_2,
	\]
	and, if
	\[
		\F^\eta
		=
		\bigl(
		F_j g_1^{(j,\eta)}
		\sqcup\cdots\sqcup
		F_j g_p^{(j,\eta)}
		\bigr)_{j\in\N},
	\]
	then
	\begin{equation}\label{distjeestmaly}
		\dist\bigl(\mu,\hat\omega_{\F^\eta}(x)\bigr)<\eta .
	\end{equation}
	Here
	\[
		\dist\bigl(\mu,\hat\omega_{\F^\eta}(x)\bigr)
		=
		\inf\{\di(\mu,\nu):\nu\in\hat\omega_{\F^\eta}(x)\},
	\]
	where $\di$ denotes the Prokhorov metric. Since each $\F^\eta$ is obtained
	from $\F$ by taking a fixed finite number of disjoint right translates at
	each level, it is again a F{\o}lner sequence.

	We claim that every measure arising as a weak$^*$ accumulation point along
	one of these auxiliary F{\o}lner sequences is the same. Let
	\[
		\F^{(1)}
		=
		\bigl(
		F_j a_1^{(j)}
		\sqcup\cdots\sqcup
		F_j a_p^{(j)}
		\bigr)_{j\in\N}
	\]
	and
	\[
		\F^{(2)}
		=
		\bigl(
		F_j b_1^{(j)}
		\sqcup\cdots\sqcup
		F_j b_q^{(j)}
		\bigr)_{j\in\N}
	\]
	be two F{\o}lner sequences of this form, with all unions disjoint, and let
	\[
		\nu_1\in\hat\omega_{\F^{(1)}}(x),
		\qquad
		\nu_2\in\hat\omega_{\F^{(2)}}(x).
	\]
	Thus $x$ quasi-generates $\nu_i$ along $\F^{(i)}$, for $i=1,2$.

	Fix a cylinder $[w]\subset \alf^G$ and $\eps>0$. By assumption, there exists
	$K\in\N$ such that the numbers
	\[
		\m(x,F_k g)([w]),
		\qquad g\in G,\ k\ge K,
	\]
	are pairwise $\eps$-close.

	For $r\ge K$, the disjointness of the union defining $\F^{(1)}_r$ gives
	\[
		\m(x,\F^{(1)}_r)([w])
		=
		\frac{1}{p}
		\sum_{i=1}^p
		\m(x,F_r a_i^{(r)})([w]).
	\]
	Similarly, for $s\ge K$,
	\[
		\m(x,\F^{(2)}_s)([w])
		=
		\frac{1}{q}
		\sum_{i=1}^q
		\m(x,F_s b_i^{(s)})([w]).
	\]
	All terms appearing in these two averages are pairwise $\eps$-close.
	Consequently, for all $r,s\ge K$,
	\[
		\Bigl|
		\m(x,\F^{(1)}_r)([w])
		-
		\m(x,\F^{(2)}_s)([w])
		\Bigr|
		<\eps .
	\]

	Choose subsequences $(r_\ell)_{\ell\in\N}$ and $(s_\ell)_{\ell\in\N}$ such that
	\[
		\m(x,\F^{(1)}_{r_\ell})
		\xrightarrow[\ell\to\infty]{w^*}
		\nu_1
		\quad\text{and}\quad
		\m(x,\F^{(2)}_{s_\ell})
		\xrightarrow[\ell\to\infty]{w^*}
		\nu_2 .
	\]
	Since $[w]$ is clopen, passing to the limit yields
	\[
		|\nu_1([w])-\nu_2([w])|\le \eps .
	\]
	As $\eps>0$ was arbitrary, we get
	\[
		\nu_1([w])=\nu_2([w]).
	\]
	Cylinders form a countable clopen basis for $\alf^G$ and determine Borel
	probability measures. Hence $\nu_1=\nu_2$.

	It follows that every set $\hat\omega_{\F^\eta}(x)$ appearing in the
	approximation above consists of the same single measure; denote it by $\nu$.
	From~\eqref{distjeestmaly}, for every $\eta>0$ we have
	\[
		\di(\mu,\nu)<2\eta .
	\]
	Thus $\mu=\nu$. Since $\mu\in\M_G(x)$ was arbitrary, $\M_G(x)$ is a
	singleton. Therefore $x$ is uniquely ergodic.
\end{proof}

We recall that a Cortez–Petite Følner sequence induces tilings of the group by translates of Følner sets, which we call \emph{principal copies} (see Figure~\ref{fig:principal}, Definition~\ref{df:principalcopy}).

\begin{figure}[h]
    \centering
    \begin{tikzpicture}[font=\small, scale=0.7]
        
        \tikzset{
            folner/.style={draw=blue!70, line width=1pt},
            shiftset/.style={draw=green!60, line width=1pt},
            shiftset_princ/.style={draw=blue!60, line width=1pt},
            shiftset_bad/.style={draw=red!60, line width=1pt},
        }
        
        \draw[step=0.5cm, gray!20, thin] (-8,-5.5) grid (9,5.5);
        \draw[step=2cm, black!70, thin] (-8,-5.5) grid (9,5.5);

        \begin{scope}[shift={(0,0)}]
            \fill[pattern=north east lines,pattern color=blue!35] (0,0) rectangle (2,2);
            \draw[shiftset_princ] (0,0) rectangle (2,2);
            \node[text=blue!60!black, fill=white, inner sep=1pt, anchor=south] at (0,2.15) {$F_1 e$};
        \end{scope}

        \begin{scope}[shift={(4,2)}]
            \fill[green!18] (0,0) rectangle (2,2);
            \draw[shiftset] (0,0) rectangle (2,2);
            \node[text=green!60!black, fill=white, inner sep=1pt, anchor=south] at (0,2.15) {$F_1 h_1$};
        \end{scope}
        \begin{scope}[shift={(-4,2)}]
            \fill[green!18] (0,0) rectangle (2,2);
            \draw[shiftset] (0,0) rectangle (2,2);
            \node[text=green!60!black, fill=white, inner sep=1pt, anchor=south] at (0,2.15) {$F_1 h_2$};
        \end{scope}
        \begin{scope}[shift={(4,-4)}]
            \fill[green!18] (0,0) rectangle (2,2);
            \draw[shiftset] (0,0) rectangle (2,2);
            \node[text=green!60!black, fill=white, inner sep=1pt, anchor=south] at (0,2.15) {$F_1 h_3$};
        \end{scope}

        \begin{scope}[shift={(-3,-4)}]
            \fill[pattern=dots,pattern color=red!55] (0,0) rectangle (2,2);
            \draw[shiftset_bad] (0,0) rectangle (2,2);
            \node[text=red!60!black, fill=white, inner sep=1pt, anchor=south] at (0,2.15) {$F_1 g_1$};
        \end{scope}
        \begin{scope}[shift={(3,-1)}]
            \fill[pattern=dots,pattern color=red!55] (0,0) rectangle (2,2);
            \draw[shiftset_bad] (0,0) rectangle (2,2);
            \node[text=red!60!black, fill=white, inner sep=1pt, anchor=south] at (0,2.15) {$F_1 g_2$};
        \end{scope}

        \node[align=center,scale=0.85] at (0,-6.5) {%
            \tikz{\draw[line width=2pt,blue!60] (0,0) rectangle (0.5,0.22);\fill[pattern=north east lines,pattern color=blue!35] (0,0) rectangle (0.5,0.22);}~principal copy ($F_1 e$)
            \quad
            \tikz{\draw[green!60] (0,0) rectangle (0.5,0.22);\fill[green!18] (0,0) rectangle (0.5,0.22);}~good principal copy
            \quad
            \tikz{\draw[red!60] (0,0) rectangle (0.5,0.22);\fill[pattern=dots,pattern color=red!55] (0,0) rectangle (0.5,0.22);}~bad copy
        };
    \end{tikzpicture}
\caption{Principal copies of $F_1$ are translates $F_1h$ with $h\in H_1$
(blue and green). The red translates are not principal copies.}
\label{fig:principal}
\end{figure}
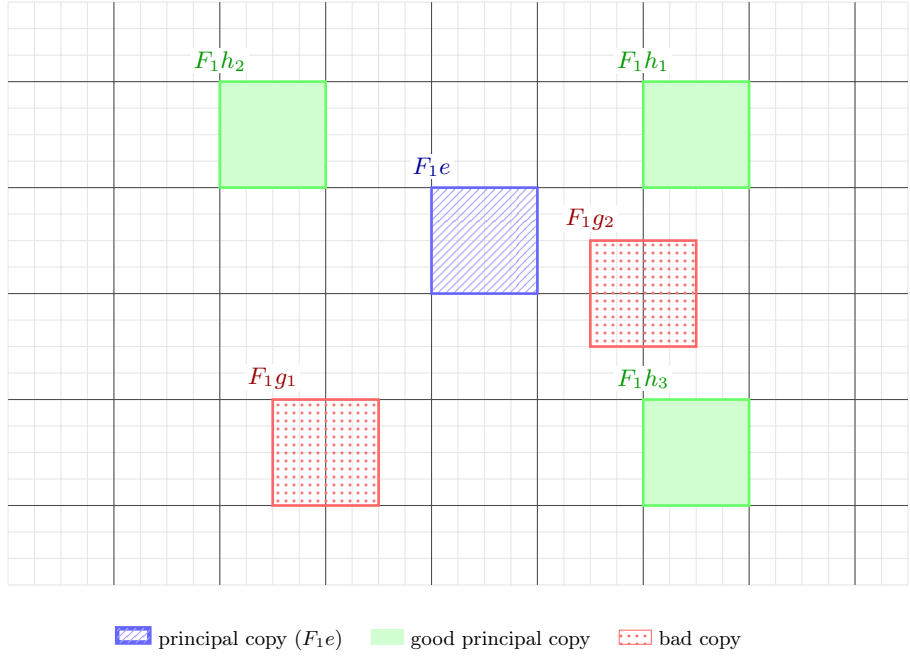

\begin{defn}\label{df:principalcopy}
	Let $\F=(F_j)_{j\in\N}$ be a Cortez–Petite Følner sequence (see Definition~\ref{Cortez-Petite}) in $G$ with associated sequence of normal subgroups $(H_j)_{j\in\N}$.  
	For $k\in\N$, a \emph{principal copy} of $F_k$ is any set of the form $F_k g$ with $g\in H_k$.
\end{defn}

From now on, we assume in addition that $G$ is residually finite. The next
lemma says that, for a Cortez--Petite F{\o}lner sequence, control of empirical
measures on principal copies of one set $F_K$ implies asymptotic control of
empirical measures on arbitrary copies of sufficiently large sets $F_n$.

\begin{lem}\label{lemma:1}
	Let $\F=(F_j)_{j\in\N}$ be a Cortez--Petite F{\o}lner sequence in $G$,
	and let $(H_j)_{j\in\N}$ be the associated sequence of normal subgroups.
	Fix $x\in\alf^G$, $K\in\N$, $w\in\alf^*$, and $\eps>0$. Suppose that all
	numbers in
	\[
		\bigl\{\m(x,F_K\gamma)([w]) : \gamma\in H_K\bigr\}
	\]
	are pairwise $\eps$-close. Then for every $\eta>0$ there exists $N\in\N$
	such that all numbers in
	\[
		\bigl\{\m(x,F_n g)([w]) : g\in G,\ n\ge N\bigr\}
	\]
	are pairwise $2(\eps+2\eta)$-close.
\end{lem}

\begin{proof}
	Fix $\eta>0$ and put
	\[
		L=F_KF_K^{-1}.
	\]
	Choose $N\in\N$ such that, for every $n\ge N$, the set $F_n$ is
	\[
		\left(L,\frac{\eta}{|F_K||L|}\right)\text{-invariant}.
	\]
	This invariance property is preserved under right translations: for every
	$g\in G$, the set $F_n g$ is also
	\[
		\left(L,\frac{\eta}{|F_K||L|}\right)\text{-invariant}.
	\]

	Fix $n\ge N$ and $g\in G$, and write
	\[
		E=F_n g .
	\]
	Since $F_K$ is a fundamental domain for $G/H_K$, the family
	\[
		\{F_K\gamma:\gamma\in H_K\}
	\]
	is a partition of $G$. Let
	\[
		\Gamma=\{\gamma\in H_K: F_K\gamma\cap E\neq\emptyset\}
	\]
	and set
	\[
		U=\bigsqcup_{\gamma\in\Gamma}F_K\gamma .
	\]
	Then $E\subset U$.

	We estimate the part of $U$ which lies outside $E$. Suppose that a tile
	$F_K\gamma$ intersects $E$ but is not contained in $E$. Choose
	$s\in F_K\gamma\cap E$ and $t\in F_K\gamma\setminus E$. Then
	\[
		t\in F_KF_K^{-1}s=Ls,
	\]
	and therefore $Ls\nsubseteq E$. Since distinct tiles are disjoint, the number
	of such bad tiles is bounded by
	\[
		|\{s\in E:Ls\nsubseteq E\}|.
	\]
	By Lemma~\ref{Z}, applied with $\eps=\eta/|F_K|$, we have
	\[
		|\{s\in E:Ls\nsubseteq E\}|
		\le
		\frac{\eta}{|F_K|}|E|.
	\]
	Each bad tile has cardinality $|F_K|$, hence
	\[
		|U\setminus E|\le \eta |E|.
	\]

	It follows that
	\[
		\bigl|
		\m(x,E)([w])-\m(x,U)([w])
		\bigr|
		\le 2\eta .
	\]
	Equivalently,
	\[
		\bigl|
		\m(x,F_n g)([w])-\m(x,U)([w])
		\bigr|
		\le 2\eta .
	\]

	On the other hand, since $U$ is a disjoint union of principal copies of
	$F_K$, we have
	\[
		\m(x,U)([w])
		=
		\frac{1}{|\Gamma|}
		\sum_{\gamma\in\Gamma}
		\m(x,F_K\gamma)([w]).
	\]
	Since all principal-copy frequencies are pairwise $\eps$-close, this average
	is $\eps$-close to every number of the form
	\[
		\m(x,F_K\gamma)([w]),\qquad \gamma\in H_K.
	\]
	Combining this with the previous estimate, we obtain that for every
	$n\ge N$, every $g\in G$, and every $\gamma\in H_K$,
	\[
		\bigl|
		\m(x,F_n g)([w])
		-
		\m(x,F_K\gamma)([w])
		\bigr|
		<
		\eps+2\eta .
	\]

	Now take arbitrary $n_1,n_2\ge N$ and $g_1,g_2\in G$. Choosing any fixed
	$\gamma\in H_K$ and applying the last estimate twice, we get
	\[
		\begin{aligned}
		&
		\bigl|
		\m(x,F_{n_1}g_1)([w])
		-
		\m(x,F_{n_2}g_2)([w])
		\bigr|
		\\
		&\qquad\le
		\bigl|
		\m(x,F_{n_1}g_1)([w])
		-
		\m(x,F_K\gamma)([w])
		\bigr|
		+
		\bigl|
		\m(x,F_K\gamma)([w])
		-
		\m(x,F_{n_2}g_2)([w])
		\bigr|
		\\
		&\qquad<
		2(\eps+2\eta).
		\end{aligned}
	\]
	This proves the claim.
\end{proof}

Lemma~\ref{lemma:1} gives the following criterion for unique ergodicity.

\begin{cor}\label{cor:ue}
	Let $x\in\alf^G$, and let $\F=(F_j)_{j\in\N}$ be a tempered
	Cortez--Petite F{\o}lner sequence in $G$ with associated sequence
	$(H_j)_{j\in\N}$. Suppose that for every $\eps>0$ and every $w\in\alf^*$
	there exists $K\in\N$ such that the numbers
	\[
		\bigl\{\m(x,F_K g)([w]) : g\in H_K\bigr\}
	\]
	are pairwise $\eps$-close. Then $x$ is uniquely ergodic.
\end{cor}

\begin{proof}
	Fix $w\in\alf^*$ and $\delta>0$. By assumption, there exists $K\in\N$ such
	that the numbers
	\[
		\bigl\{\m(x,F_K g)([w]) : g\in H_K\bigr\}
	\]
	are pairwise $\delta/4$-close. Applying Lemma~\ref{lemma:1} with
	$\eps=\delta/4$ and $\eta=\delta/8$, we obtain $N\in\N$ such that all
	numbers in
	\[
		\bigl\{\m(x,F_n g)([w]) : g\in G,\ n\ge N\bigr\}
	\]
	are pairwise $\delta$-close.

	Since $w\in\alf^*$ and $\delta>0$ were arbitrary, the hypothesis of
	Lemma~\ref{lemma:2} is satisfied. Therefore $x$ is uniquely ergodic.
\end{proof}
Lemma~\ref{ostatnilemat} shows that if a point is generic for an ergodic
measure, then the empirical measures taken over principal copies of
F{\o}lner sets are uniformly close to that measure. The proof uses Oxtoby's
criterion for ergodicity of a generic point; we recall the version needed below.

Let
\[
	Q=\{x\in X : x \text{ is generic for some } \mu\in\M_G(X)
	\text{ with respect to } \F\}.
\]
By the pointwise ergodic theorem for tempered F{\o}lner sequences, we have
\[
	\mu(Q)=1
	\qquad\text{for every } \mu\in\M_G^e(X).
\]

For $k\in\N$, $x\in X$, and $f\in\mathcal C(X)$, define
\[
	A(f,k)(x)
	=
	\frac{1}{|F_k|}
	\sum_{g\in F_k} f(gx).
\]
If $y\in Q$, then $y$ is generic with respect to $\F$, and hence the limit
\[
	f^*(y)
	=
	\lim_{k\to\infty} A(f,k)(y)
\]
exists for every $f\in\mathcal C(X)$.

\begin{thm}[Oxtoby criterion for ergodicity in amenable groups, {\cite{Oxtoby52,LM}}]\label{thm:Oxtoby}
	Suppose that $x\in X$ is generic for some $\mu\in \M_G(X)$ with respect to
	a F{\o}lner sequence $\F=(F_n)_{n\in\N}$. Then the following are equivalent:
	\begin{enumerate}
		\item $\mu$ is ergodic;
		\item for every $f\in\mathcal C(X)$ and every $\alpha>0$,
		\[
			\lim_{k\to\infty}\;\limsup_{n\to\infty}
			\frac{1}{|F_n|}
			\bigl|
			\{g\in F_n : |A(f,k)(gx)-f^*(x)|>\alpha\}
			\bigr|
			=0.
		\]
	\end{enumerate}
\end{thm}

The following lemma is similar in spirit to \cite[Theorem 5.5 and Theorem 5.8]{Ergodicblockdecomp}.
\begin{lem}\label{ostatnilemat}
	Let $x$ be an ergodic point generic for a measure $\mu$ with respect to a tempered Cortez--Petite Følner sequence $\F=(F_n)_{n\in\N}$.  
	Fix $w_1,w_2,\ldots,w_p\in\alf^*$ and $\eps>0$. Then for all sufficiently large $M\in\N$ the following holds:  
	
	for every $m\geq M$ there exists a set $\mathcal{S}$ of at least $(1-\eps)|F_m|/|F_M|$ principal copies of $F_M$ inside $F_m$ such that, for each word $w_i$ with $1\leq i\leq p$ and for every $S\in\mathcal{S}$, the values $\m(x,S)([w_i])$ and $\mu([w_i])$ are $\eps$-close.
\end{lem}

\begin{proof}
	It is enough to prove the statement for a smaller value of $\eps$, so we
	may assume that $0<\eps<1$.

	For $i=1,\dots,p$, let
	\[
		\varphi_i=\mathbf 1_{[w_i]}.
	\]
	Since each cylinder $[w_i]$ is clopen, the functions $\varphi_i$ are
	continuous. Since $x$ is generic for $\mu$ with respect to $\F$, we have
	\[
		\varphi_i^*(x)=\mu([w_i])
		\qquad\text{for } i=1,\dots,p.
	\]

	Set
	\[
		\tau=\frac{\eps}{4}
		\qquad\text{and}\qquad
		\theta=\frac{\eps}{4}.
	\]
	Choose $\eta>0$ so small that
	\begin{equation}\label{eq:eta-choice}
		\frac{\eta}{\tau^2}+\eta<\theta\eps
		\qquad\text{and}\qquad
		\tau+\theta+\eta<\eps .
	\end{equation}

	By Theorem~\ref{thm:Oxtoby}, and since the functions $\varphi_i$ are bounded,
	we may choose $k\in\N$ such that
	\begin{equation}\label{eq:oxtoby-square}
		\limsup_{n\to\infty}
		\frac{1}{|F_n|}
		\sum_{a\in F_n}\sum_{i=1}^p
		\bigl(A(\varphi_i,k)(a x)-\mu([w_i])\bigr)^2
		<\eta .
	\end{equation}
	Indeed, this follows from the Oxtoby criterion by first controlling the
	sets on which the deviations are larger than a small threshold, and then
	using boundedness.

	Put
	\[
		L=F_kF_k^{-1}.
	\]
	Choose $M_0\ge k$ such that for every $n\ge M_0$ we have
	\begin{equation}\label{eq:square-large-n}
		\frac{1}{|F_n|}
		\sum_{a\in F_n}\sum_{i=1}^p
		\bigl(A(\varphi_i,k)(a x)-\mu([w_i])\bigr)^2
		<\eta ,
	\end{equation}
	and such that $F_n$ is
	\[
		\left(L,\frac{\eta}{|F_k||L|}\right)\text{-invariant}.
	\]

	We show that every $M\ge M_0$ satisfies the conclusion of the lemma. Fix
	such an $M$ and let $m\ge M$. By the Cortez--Petite decomposition, iterated
	if necessary,
	\[
		F_m=\bigsqcup_{\gamma\in H_k\cap F_m}F_k\gamma.
	\]
	Equivalently,
	\[
		F_m=\bigsqcup_{f\in F_k} f(H_k\cap F_m).
	\]
	Using~\eqref{eq:square-large-n} and averaging over $f\in F_k$, we can choose
	$f\in F_k$ such that
	\begin{equation}\label{eq:good-f-square}
		\frac{|F_k|}{|F_m|}
		\sum_{\gamma\in H_k\cap F_m}\sum_{i=1}^p
		\bigl(A(\varphi_i,k)(f\gamma x)-\mu([w_i])\bigr)^2
		<\eta .
	\end{equation}

	Let $T\subset H_k\cap F_m$ be the set of all $\gamma$ such that
	\[
		\bigl|A(\varphi_i,k)(f\gamma x)-\mu([w_i])\bigr|< \tau
		\qquad\text{for every } i=1,\dots,p.
	\]
	By the Chebyshev--Markov inequality applied to~\eqref{eq:good-f-square},
	\[
		|T|
		\ge
		\left(1-\frac{\eta}{\tau^2}\right)\frac{|F_m|}{|F_k|}.
	\]

	For this fixed $f$, the family
	\[
		\{F_k f\gamma:\gamma\in H_k\}
	\]
	is a partition of $G$. Indeed, since $H_k$ is normal, this family is obtained
	from the partition $\{F_k\gamma:\gamma\in H_k\}$ by right multiplication by
	$f$, after reindexing.

	We now discard the cells which are not contained in $F_m$. Let
	\[
		T_{\mathrm{in}}
		=
		\{\gamma\in T:F_k f\gamma\subset F_m\}.
	\]
	We claim that
	\[
		|T_{\mathrm{in}}|
		\ge
		\left(1-\frac{\eta}{\tau^2}-\eta\right)\frac{|F_m|}{|F_k|}.
	\]
	Indeed, suppose that a cell $F_k f\gamma$ intersects $F_m$ but is not
	contained in $F_m$. Choose
	\[
		s\in F_k f\gamma\cap F_m
		\qquad\text{and}\qquad
		t\in F_k f\gamma\setminus F_m.
	\]
	Then
	\[
		t\in F_kF_k^{-1}s=Ls,
	\]
	and hence $Ls\nsubseteq F_m$. Since distinct cells are disjoint, the number
	of such boundary cells is at most
	\[
		|\{s\in F_m:Ls\nsubseteq F_m\}|.
	\]
	By Lemma~\ref{Z}, applied with $\eps=\eta/|F_k|$, this number is at most
	\[
		\frac{\eta}{|F_k|}|F_m|.
	\]
	The claim follows.

	Set
	\[
		Z=\bigsqcup_{\gamma\in T_{\mathrm{in}}}F_k f\gamma.
	\]
	Then $Z\subset F_m$ and
	\[
		|Z|
		\ge
		\left(1-\frac{\eta}{\tau^2}-\eta\right)|F_m|.
	\]

	Now decompose $F_m$ into principal copies of $F_M$:
	\[
		F_m=\bigsqcup_{h\in H_M\cap F_m}F_Mh.
	\]
	Call a principal copy $S=F_Mh$ \emph{good} if
	\[
		|S\cap Z|\ge (1-\theta)|S|.
	\]
	Since
	\[
		|F_m\setminus Z|
		\le
		\left(\frac{\eta}{\tau^2}+\eta\right)|F_m|,
	\]
	the number of principal copies which are not good is at most
	\[
		\frac{\eta/\tau^2+\eta}{\theta}\cdot\frac{|F_m|}{|F_M|}.
	\]
	By~\eqref{eq:eta-choice}, this is less than
	\[
		\eps\frac{|F_m|}{|F_M|}.
	\]
	Thus there are at least
	\[
		(1-\eps)\frac{|F_m|}{|F_M|}
	\]
	good principal copies of $F_M$ inside $F_m$.

	It remains to show that every good principal copy has the desired empirical
	frequencies. Fix a good copy
	\[
		S=F_Mh.
	\]
	Since $M\ge M_0$, the set $F_M$ is
	\[
		\left(L,\frac{\eta}{|F_k||L|}\right)\text{-invariant},
	\]
	and the same is true for its right translate $S$.

	Let
	\[
		Y_S
		=
		\bigsqcup\{F_k f\gamma:
		\gamma\in T_{\mathrm{in}},\ F_k f\gamma\subset S\}.
	\]
	The only cells from the partition $\{F_k f\gamma:\gamma\in H_k\}$ which can
	contribute to $S\cap Z$ but not to $Y_S$ are cells which intersect $S$
	without being contained in $S$. As above, the number of such boundary cells
	is at most
	\[
		|\{s\in S:Ls\nsubseteq S\}|
		\le
		\frac{\eta}{|F_k|}|S|.
	\]
	Hence their total size is at most $\eta |S|$. Since $S$ is good,
	\[
		|Y_S|
		\ge
		|S\cap Z|-\eta |S|
		\ge
		(1-\theta-\eta)|S|.
	\]

	Fix $i\in\{1,\dots,p\}$. For every cell $F_k f\gamma$ appearing in $Y_S$,
	we have $\gamma\in T$, and therefore
	\[
		\m(x,F_k f\gamma)([w_i])
		=
		A(\varphi_i,k)(f\gamma x)
	\]
	is $\tau$-close to $\mu([w_i])$. Since $Y_S$ is a disjoint union of such
	cells, we get
	\[
		\bigl|\m(x,Y_S)([w_i])-\mu([w_i])\bigr|\le\tau.
	\]
	Moreover,
	\[
		|S\setminus Y_S|\le(\theta+\eta)|S|.
	\]
	Removing a subset of relative size at most $\rho$ changes the frequency of a
	set by at most $\rho$. Hence
	\[
		\bigl|\m(x,S)([w_i])-\m(x,Y_S)([w_i])\bigr|
		\le
		\theta+\eta .
	\]
	Consequently,
	\[
		\bigl|\m(x,S)([w_i])-\mu([w_i])\bigr|
		\le
		\tau+\theta+\eta
		<\eps
	\]
	by~\eqref{eq:eta-choice}. This holds for every $i=1,\dots,p$.

	Taking $\mathcal S$ to be the family of good principal copies completes the
	proof.
\end{proof}

We shall use the ergodic decomposition in the following standard form. In the
application below, it allows us to decompose a $G$-ergodic measure into finitely
many $H$-ergodic components, where $H$ is a finite-index normal subgroup.

\begin{thm}[Ergodic decomposition, {\cite{Glasner}}]\label{thm:ergodic-decomposition}
	Let $\mathbf X=(X,\mathcal X,G,\mu)$ be a measure-preserving system, and let
	$\mathcal J\subset\mathcal X$ denote the $\sigma$-algebra of
	$G$-invariant sets. Let
	\[
		\pi:\mathbf X\to\mathbf Y
	\]
	be the factor corresponding to $\mathcal J$. Write the disintegration of
	$\mu$ over the factor measure $\nu$ as
	\[
		\mu=\int_Y \mu_y\,d\nu(y).
	\]
	For $y\in Y$, denote
	\[
		X_y=\pi^{-1}(y)
		\qquad\text{and}\qquad
		\mathcal X_y=\mathcal X\cap X_y .
	\]
	Then:
	\begin{enumerate}
		\item\label{identity}
		the group $G$ acts trivially on $Y$;
		\item
		for $\nu$-almost every $y\in Y$, the system
		\[
			(X_y,\mathcal X_y,\mu_y,G)
		\]
		is ergodic.
	\end{enumerate}
\end{thm}

\begin{lem}\label{almostergodic}
	Let $(X,G,\mu)$ be an ergodic measure-preserving system, and let
	$H\lhd G$ be a normal subgroup of finite index. Let
	$\F=(F_n)_{n\in\N}$ be a tempered \Folner sequence in $G$. After discarding
	finitely many initial terms, set
	\[
		\F^H=(F_n\cap H)_{n\in\N}.
	\]
	Then $\F^H$ is a tempered \Folner sequence in $H$. Moreover, for the
	restricted system $(X,H,\mu)$, there exist $H$-ergodic measures
	$(\mu_s)_{s\in G/H}$ such that:
	\begin{enumerate}
		\item
		\[
			\mu=\frac{1}{|G/H|}\sum_{s\in G/H}\mu_s;
		\]
		\item\label{equivariance}
		\[
			g_*\mu_s=\mu_{gs}
			\qquad\text{for all } g\in G \text{ and } s\in G/H;
		\]
		\item\label{empiric}
		for $\mu$-almost every $x\in X$, there exists $s\in G/H$ such that
		\[
			\hat\omega_{\F^H}(x)=\{\mu_s\}.
		\]
	\end{enumerate}
	The measures $\mu_s$ need not be distinct.
\end{lem}

\begin{proof}
	Since $\mu$ is $G$-invariant, it is also $H$-invariant, although it need not
	be $H$-ergodic. Let $\mathcal J_G$ and $\mathcal J_H$ denote the
	$\sigma$-algebras of $G$-invariant and $H$-invariant sets, respectively.
	Then
	\[
		\mathcal J_G\subseteq \mathcal J_H .
	\]
	Let
	\[
		\pi:(X,H,\mu)\to(Y,\nu)
	\]
	be the factor corresponding to $\mathcal J_H$. By
	Theorem~\ref{thm:ergodic-decomposition}, the action of $H$ on $Y$ is
	trivial, and the conditional measures in the disintegration
	\[
		\mu=\int_Y \mu_y\,d\nu(y)
	\]
	are $H$-ergodic for $\nu$-almost every $y\in Y$.

	Because $H$ is normal, the action of $G$ sends $H$-invariant sets to
	$H$-invariant sets. Hence the $G$-action descends to an action on $Y$.
	Since $H$ acts trivially on $Y$, this action factors through the finite
	quotient $G/H$. Moreover, since $(X,G,\mu)$ is ergodic, the factor
	$(Y,G,\nu)$ is ergodic. Therefore $\nu$ is supported on a single finite
	$G/H$-orbit, and the measure on this orbit is uniform.

	Choose a point $y_0$ in this orbit. For a coset $s\in G/H$, choose any
	representative $g\in G$ of $s$ and define
	\[
		\mu_s=g_*\mu_{y_0}.
	\]
	This gives a family indexed by $G/H$, possibly with repetitions. Indeed,
	if two cosets send $y_0$ to the same point of the orbit, then the
	corresponding conditional measures agree. Each $\mu_s$ is $H$-ergodic,
	because it is the image of an $H$-ergodic conditional measure under an
	element of $G$.

	Since the measure on the finite orbit is uniform, and since indexing by
	$G/H$ only repeats each point of the orbit the same number of times, we get
	\[
		\mu=\frac{1}{|G/H|}\sum_{s\in G/H}\mu_s.
	\]
	The definition also gives
	\[
		g_*\mu_s=\mu_{gs}
		\qquad\text{for all } g\in G \text{ and } s\in G/H,
	\]
	which proves~\ref{equivariance}.

	It remains to prove~\ref{empiric}. Since $H$ has finite index in $G$, the
	sets $F_n\cap H$ are non-empty for all sufficiently large $n$, and
	\[
		\frac{|F_n\cap H|}{|F_n|}\longrightarrow \frac{1}{|G/H|}.
	\]
	Moreover, $\F^H=(F_n\cap H)_{n\in\N}$ is a \Folner sequence in $H$. It is
	also tempered: indeed,
	\[
		\bigcup_{k\le n}(F_k\cap H)^{-1}(F_{n+1}\cap H)
		\subseteq
		\bigcup_{k\le n}F_k^{-1}F_{n+1},
	\]
	and the temperedness of $\F$, together with
	$|F_{n+1}\cap H|\asymp |F_{n+1}|$, gives the required bound after discarding
	finitely many initial terms.

	Applying the pointwise ergodic theorem to the $H$-action and the tempered
	\Folner sequence $\F^H$, we obtain that for each $s\in G/H$ and for
	$\mu_s$-almost every $x\in X$,
	\[
		\m(x,F_n\cap H)\xrightarrow[n\to\infty]{w^*}\mu_s.
	\]
	Equivalently,
	\[
		\hat\omega_{\F^H}(x)=\{\mu_s\}.
	\]
	Integrating over the finite decomposition of $\mu$ gives the same conclusion
	for $\mu$-almost every $x$. This proves~\ref{empiric}.
\end{proof}
\begin{defn}\label{piece}
	Let $m\in\N$ and let
	\[
		\mathcal C_m=\alf^{F_m}
	\]
	be the set of all configurations on $F_m$. Fix a map
	\[
		\psi\colon \mathcal C_m\to\mathcal C_m .
	\]
	It induces a block map $\Psi\colon\alf^G\to\alf^G$ as follows. For
	$x\in\alf^G$, $h\in H_m$, and $a\in F_m$, set
	\[
		\Psi(x)_{ah}
		=
		\psi\bigl((x_{bh})_{b\in F_m}\bigr)(a).
	\]
	Thus $\Psi$ acts independently on the principal copies $F_mh$, $h\in H_m$.

	If $X\subseteq\alf^G$ is a subsystem and $\Psi(X)\subseteq X$, then the
	restriction $\Psi\colon X\to X$ is called a \emph{jigsaw of level $m$}.
	A jigsaw is always understood together with its chosen level, and we write
	$\ell(\Psi)=m$.
\end{defn}

The next lemma provides a full-measure set of points whose generic behaviour is
stable under finite compositions of jigsaws, provided that the averaging is
taken along a subgroup level below all levels involved in the composition.

\begin{lem}[Universally good points]\label{lem:universal}
	Let $X\subseteq\alf^G$ be a subsystem, and let $\mu\in\M_G^e(X)$.
	Let $(H_n)_{n\ge 1}$ be a decreasing sequence of finite-index normal
	subgroups with
	\[
		\bigcap_{n\ge 1}H_n=\{e\},
	\]
	and let $\F=(F_n)_{n\ge 1}$ be the associated tempered Cortez--Petite
	\Folner sequence, where each $F_n$ is a fundamental domain for $G/H_n$.

	For each $n$, let
	\[
		\F^{H_n}=(F_k\cap H_n)_{k\ge 1}
	\]
	denote the induced \Folner sequence in $H_n$, after discarding finitely many
	initial empty terms.  

	There exists a $\mu$-conull set $A\subseteq X$ such that every $x\in A$
	satisfies the following properties:
	\begin{enumerate}
		\item[\textup{(U1)}]
		$x$ is $\F$-generic for $\mu$.

		\item[\textup{(U2)}]
		For every $n\ge 1$, there exists an $H_n$-ergodic component
		$\mu_{n,s}$ of the disintegration of $\mu$ relative to the action of
		$H_n$ such that
		\[
			\hat\omega_{\F^{H_n}}(x)=\{\mu_{n,s}\}.
		\]

		\item[\textup{(U3)}]
		If
		\[
			\Phi=\Psi_r\circ\cdots\circ\Psi_1
		\]
		is any finite composition of jigsaws and
		\[
			m>\max\{\ell(\Psi_1),\dots,\ell(\Psi_r)\},
		\]
		then
		\[
			\hat\omega_{\F^{H_m}}(\Phi(x))
			=
			\{\Phi_*(\mu_{m,s})\},
		\]
		where $\mu_{m,s}$ is the component from \textup{(U2)} corresponding to
		$x$ and $m$. In particular, $\Phi(x)$ is $\F^{H_m}$-generic for an
		$H_m$-ergodic component of $\Phi_*\mu$.
	\end{enumerate}
\end{lem}

\begin{proof}
	For each $n\ge 1$, Lemma~\ref{almostergodic}, applied to the finite-index
	normal subgroup $H_n$, gives a decomposition
	\[
		\mu=\frac{1}{|G/H_n|}\sum_{s\in G/H_n}\mu_{n,s}
	\]
	into $H_n$-ergodic measures, possibly with repetitions. It also gives a
	$\mu$-conull set $A_n\subseteq X$ such that every $x\in A_n$ satisfies
	\[
		\hat\omega_{\F^{H_n}}(x)=\{\mu_{n,s}\}
	\]
	for some $s\in G/H_n$.

	Since $\mu$ is $G$-ergodic and $\F$ is tempered, the pointwise ergodic
	theorem gives a $\mu$-conull set $A_0\subseteq X$ of $\F$-generic points
	for $\mu$.

	We now record the equivariance property of jigsaws. Let $\Psi$ be a jigsaw
	of level $n$. By definition, $\Psi$ acts independently on the principal
	copies
	\[
		F_nh,\qquad h\in H_n.
	\]
	If $h_0\in H_n$, then the shift by $h_0$ permutes these copies, and the same
	block rule is applied on each copy. Hence
	\[
		\Psi(h_0x)=h_0\Psi(x)
		\qquad\text{for all } h_0\in H_n \text{ and } x\in X.
	\]
	Thus $\Psi$ is $H_n$-equivariant. Since the subgroups $(H_n)_{n\in\N}$ are
	decreasing, $\Psi$ is also $H_m$-equivariant for every $m\ge n$.

	It follows that if
	\[
		\Phi=\Psi_r\circ\cdots\circ\Psi_1
	\]
	and
	\[
		m>\max\{\ell(\Psi_1),\dots,\ell(\Psi_r)\},
	\]
	then $\Phi$ is $H_m$-equivariant.

	We next observe that genericity along $\F^{H_m}$ is preserved under such an
	$H_m$-equivariant map. Suppose that
	\[
		\hat\omega_{\F^{H_m}}(x)=\{\mu_{m,s}\}.
	\]
	For all sufficiently large $q$, write
	\[
		\omega_q^{H_m}(x)
		=
		\frac{1}{|F_q\cap H_m|}
		\sum_{h\in F_q\cap H_m}\delta_{hx}.
	\]
	Then
	\[
		\omega_q^{H_m}(x)\xrightarrow[q\to\infty]{w^*}\mu_{m,s}.
	\]
	Using the $H_m$-equivariance of $\Phi$, we obtain
	\[
		\omega_q^{H_m}(\Phi(x))
		=
		\frac{1}{|F_q\cap H_m|}
		\sum_{h\in F_q\cap H_m}\delta_{h\Phi(x)}
		=
		\frac{1}{|F_q\cap H_m|}
		\sum_{h\in F_q\cap H_m}\delta_{\Phi(hx)}
		=
		\Phi_*\omega_q^{H_m}(x).
	\]
	Since $\Phi$ is continuous, pushforward by $\Phi$ is continuous in the
	weak$^*$ topology. Therefore
	\[
		\omega_q^{H_m}(\Phi(x))
		=
		\Phi_*\omega_q^{H_m}(x)
		\xrightarrow[q\to\infty]{w^*}
		\Phi_*\mu_{m,s}.
	\]
	Hence
	\[
		\hat\omega_{\F^{H_m}}(\Phi(x))=\{\Phi_*\mu_{m,s}\}.
	\]
	Moreover, $\Phi_*\mu_{m,s}$ is $H_m$-ergodic, because it is the pushforward
	of an $H_m$-ergodic measure under an $H_m$-equivariant map.

	Finally, set
	\[
		A=\bigcap_{n\ge 0}A_n.
	\]
	Then $\mu(A)=1$. Property \textup{(U1)} holds by the definition of $A_0$,
	and property \textup{(U2)} holds by the definition of the sets $A_n$ for
	$n\ge 1$. Property \textup{(U3)} follows from the preceding paragraph,
	applied with the corresponding value of $m$.
\end{proof}

\begin{defn}
	A point $x\in X$ is called \emph{universally good} if it satisfies
	\textup{(U1)}--\textup{(U3)} in Lemma~\ref{lem:universal}.
\end{defn}

\begin{rem}
	In particular, every point which is generic for a measure whose restriction
	to each $H_n$ is ergodic is universally good.
\end{rem}

\begin{lem}\label{nietracimy}
	Let $X\subseteq\alf^G$ be a subsystem, and let $x\in X$ be a universally
	good point which is generic for some $\mu\in\M_G^e(X)$ with respect to
	$\F$. Let $\Psi\colon X\to X$ be a jigsaw of level $n$, induced by a block
	map
	\[
		\psi\colon\mathcal C_n\to\mathcal C_n,
	\]
	and set
	\[
		y=\Psi(x).
	\]
	Then $y$ is a universally good point and is generic, with respect to $\F$,
	for some measure $\nu\in\M_G^e(X)$.
\end{lem}

\begin{proof}
	By Lemma~\ref{almostergodic}, applied to the subgroup $H_n$, the measure
	$\mu$ decomposes as
	\[
		\mu=\frac{1}{|G/H_n|}\sum_{s\in G/H_n}\mu_{n,s},
	\]
	where the measures $\mu_{n,s}$ are $H_n$-ergodic, possibly with
	repetitions. Since $x$ is universally good, there is some $p\in G/H_n$ such
	that
	\[
		\hat\omega_{\F^{H_n}}(x)=\{\mu_{n,p}\}.
	\]
	Set
	\[
		\lambda=\Psi_*\mu_{n,p}.
	\]
	Since $\Psi$ is $H_n$-equivariant and $\mu_{n,p}$ is $H_n$-ergodic, the
	measure $\lambda$ is $H_n$-ergodic. Moreover, as in the proof of
	Lemma~\ref{lem:universal}, the $H_n$-equivariance and continuity of $\Psi$
	give
	\[
		\hat\omega_{\F^{H_n}}(y)=\{\lambda\}.
	\]

	Choose the fundamental domain $F_n$ as a set of representatives for
	$G/H_n$ and define
	\[
		\nu=\frac{1}{|F_n|}\sum_{s\in F_n}s_*\lambda.
	\]
	This measure is $G$-invariant. Indeed, multiplication by any element of
	$G$ permutes the cosets of $H_n$, and $\lambda$ is $H_n$-invariant.
	Moreover, $\nu$ is $G$-ergodic. If $B$ is $G$-invariant, then
	\[
		\nu(B)
		=
		\frac{1}{|F_n|}\sum_{s\in F_n}\lambda(s^{-1}B)
		=
		\lambda(B).
	\]
	Since $B$ is also $H_n$-invariant and $\lambda$ is $H_n$-ergodic, we get
	$\lambda(B)\in\{0,1\}$. Hence $\nu(B)\in\{0,1\}$.

	We now prove that $y$ is generic for $\nu$ along $\F$. Let
	$f\in\mathcal C(X)$. For $M>n$, put
	\[
		C_M=F_M\cap H_n .
	\]
	By the Cortez--Petite construction,
	\[
		F_M=\bigsqcup_{s\in F_n}sC_M .
	\]
	Therefore
	\[
		\frac{1}{|F_M|}\sum_{g\in F_M}f(gy)
		=
		\frac{1}{|F_n|}\sum_{s\in F_n}
		\frac{1}{|C_M|}\sum_{h\in C_M}f(shy).
	\]
	Since $\Psi$ is $H_n$-equivariant and $y=\Psi(x)$, we have
	\[
		hy=h\Psi(x)=\Psi(hx)
		\qquad\text{for every } h\in H_n .
	\]
	Hence
	\[
		\frac{1}{|F_M|}\sum_{g\in F_M}f(gy)
		=
		\frac{1}{|F_n|}\sum_{s\in F_n}
		\frac{1}{|C_M|}\sum_{h\in C_M}
		(U_s f)(\Psi(hx)),
	\]
	where $U_s f(z)=f(sz)$. Since $x$ is $\F^{H_n}$-generic for
	$\mu_{n,p}$, and since $U_s f\circ\Psi$ is continuous, the inner averages
	converge to
	\[
		\int_X U_s f\circ\Psi\,d\mu_{n,p}
		=
		\int_X U_s f\,d\lambda
		=
		\int_X f\,d(s_*\lambda).
	\]
	It follows that
	\[
		\lim_{M\to\infty}
		\frac{1}{|F_M|}\sum_{g\in F_M}f(gy)
		=
		\frac{1}{|F_n|}\sum_{s\in F_n}\int_X f\,d(s_*\lambda)
		=
		\int_X f\,d\nu .
	\]
	Thus $y$ is $\F$-generic for the $G$-ergodic measure $\nu$.

	It remains to check that $y$ is universally good. Property \textup{(U1)}
	has just been proved. We verify \textup{(U2)}. Fix $r\ge 1$.

	First suppose that $r\ge n$. Then $\Psi$ is $H_r$-equivariant. Since $x$
	is universally good, there is an $H_r$-ergodic component $\mu_{r,t}$ of
	$\mu$ such that
	\[
		\hat\omega_{\F^{H_r}}(x)=\{\mu_{r,t}\}.
	\]
	As in the proof of Lemma~\ref{lem:universal}, the $H_r$-equivariance and
	continuity of $\Psi$ give
	\[
		\hat\omega_{\F^{H_r}}(y)=\{\Psi_*\mu_{r,t}\}.
	\]
	The measure $\Psi_*\mu_{r,t}$ is $H_r$-ergodic.

	Now suppose that $r<n$. Let
	\[
		D_{r,n}=F_n\cap H_r .
	\]
	Then $D_{r,n}$ is a set of representatives for $H_r/H_n$, and for
	$M\ge n$ the Cortez--Petite construction gives
	\[
		F_M\cap H_r=\bigsqcup_{s\in D_{r,n}}s(F_M\cap H_n).
	\]
	Repeating the averaging argument used above to prove $\F$-genericity of
	$y$, but with $H_r$ in place of $G$, we obtain
	\[
		\hat\omega_{\F^{H_r}}(y)
		=
		\left\{
		\frac{1}{|D_{r,n}|}\sum_{s\in D_{r,n}}s_*\lambda
		\right\}.
	\]
	The measure inside the braces is $H_r$-invariant and $H_r$-ergodic by the
	same finite-orbit averaging argument used above.

	In both cases, write $\nu_r$ for the unique measure in
	$\hat\omega_{\F^{H_r}}(y)$. We have shown that $\nu_r$ is $H_r$-ergodic.
	Moreover, since $y$ is $\F$-generic for $\nu$, decomposing
	\[
		F_M=\bigsqcup_{s\in F_r}s(F_M\cap H_r)
	\]
	and passing to the limit gives
	\[
		\nu=\frac{1}{|F_r|}\sum_{s\in F_r}s_*\nu_r .
	\]
	Hence $\nu_r$ is an $H_r$-ergodic component of $\nu$. This proves
	\textup{(U2)}.

	Finally, we verify \textup{(U3)}. Let
	\[
		\Phi=\Psi_k\circ\cdots\circ\Psi_1
	\]
	be a finite composition of jigsaws, and let
	\[
		r>\max\{\ell(\Psi_1),\dots,\ell(\Psi_k)\}.
	\]
	Then $\Phi$ is $H_r$-equivariant. By \textup{(U2)}, there is an
	$H_r$-ergodic component $\nu_r$ of $\nu$ such that
	\[
		\hat\omega_{\F^{H_r}}(y)=\{\nu_r\}.
	\]
	As in the proof of Lemma~\ref{lem:universal}, the $H_r$-equivariance and
	continuity of $\Phi$ imply
	\[
		\hat\omega_{\F^{H_r}}(\Phi(y))=\{\Phi_*\nu_r\}.
	\]
	Moreover, $\Phi_*\nu_r$ is $H_r$-ergodic. Hence \textup{(U3)} holds.

	Therefore $y$ satisfies \textup{(U1)}--\textup{(U3)} and is generic along
	$\F$ for the ergodic measure $\nu$. Thus $y$ is universally good.
\end{proof}

	\begin{rem}
	The assumption of universal goodness in Lemma~\ref{nietracimy} cannot be
	replaced merely by genericity for an ergodic measure. Let $G=\mathbb Z$ and
	consider a point obtained by concatenating blocks
	\[
		(01)^{k_1},\ (10)^{k_2},\ (01)^{k_3},\ (10)^{k_4},\ldots,
		\qquad k_n=n^n.
	\]
	With respect to the usual shift action, this point is generic for the
	uniform measure on the period-two orbit
	\[
		\{(01)^\infty,(10)^\infty\}.
	\]
	Indeed, away from the sparse block boundaries, the orbit spends asymptotically
	equal time in the two phases.

	However, this point is not generic for the restricted action of
	$2\mathbb Z$. Along the even iterates, the phase seen inside a block
	$(01)^{k_n}$ differs from the phase seen inside a block $(10)^{k_{n+1}}$.
	Since the block lengths grow very fast, the empirical measures along
	$2\mathbb Z$ have different accumulation points.

	Now consider the jigsaw of level $2$ which sends the block $01$ to $11$ and
	acts as the identity on the other two-letter blocks. Under this map, the
	blocks $(01)^{k_n}$ are changed into $(11)^{k_n}$, while the blocks
	$(10)^{k_{n+1}}$ remain unchanged. Hence, in the image point, the frequency
	of the symbol $0$ oscillates between values close to $0$ and values close to
	$1/2$. In particular, the image point is not generic.

	This shows that the additional genericity requirements included in universal
	goodness are essential.
\end{rem}

\subsection{The main construction}

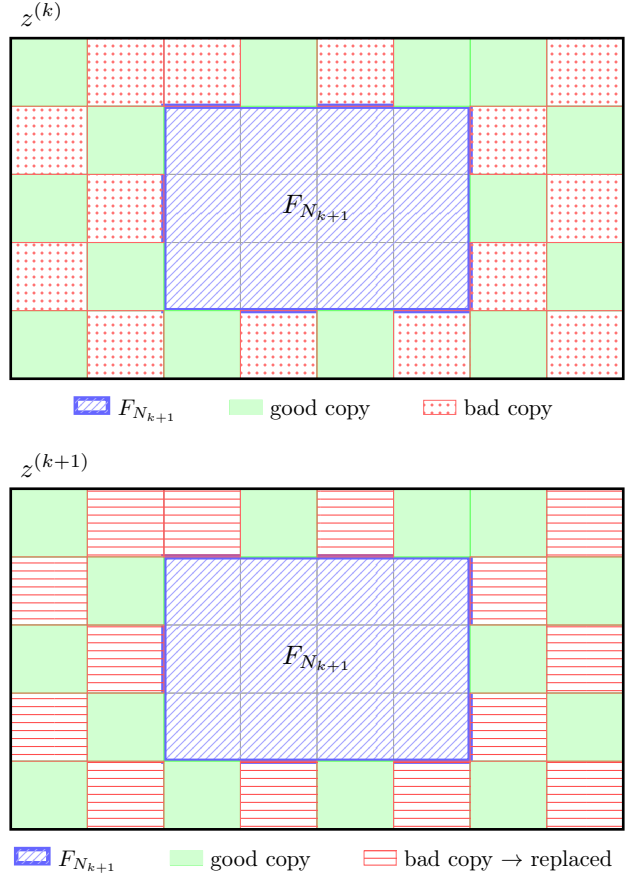
\begin{figure}[t]
    \centering
    
    \begin{tikzpicture}[>=Latex,scale=0.9]
        \def\W{1.125}
        \def\H{1.0}
        
        \foreach \c in {0,...,7}{
            \foreach \r in {0,...,4}{
                \draw[very thin,draw=black!35] ({\c*\W},{\r*\H}) rectangle ({\c*\W+\W},{\r*\H+\H});
            }
        }
        
        \fill[pattern=north east lines,pattern color=blue!35] ({2*\W},{1*\H}) rectangle ({6*\W},{4*\H});
        \draw[line width=2pt,blue!60] ({2*\W},{1*\H}) rectangle ({6*\W},{4*\H});
        \node at ({4*\W},{2.5*\H}) {$F_{N_{k+1}}$};
        
        \foreach \c/\r in {0/0,0/2,0/4,1/1,1/3, 6/0,6/2,6/4,7/1,7/3, 2/0,4/0, 3/4,5/4}{
            \fill[green!18] ({\c*\W},{\r*\H}) rectangle ({\c*\W+\W},{\r*\H+\H});
            \draw[green!60]     ({\c*\W},{\r*\H}) rectangle ({\c*\W+\W},{\r*\H+\H});
        }
        
        \foreach \c/\r in {0/1,0/3,1/0,1/2,1/4, 6/1,6/3,7/0,7/2,7/4, 3/0,5/0, 2/4,4/4}{
            \fill[pattern=dots,pattern color=red!55] ({\c*\W},{\r*\H}) rectangle ({\c*\W+\W},{\r*\H+\H});
            \draw[red!60]                           ({\c*\W},{\r*\H}) rectangle ({\c*\W+\W},{\r*\H+\H});
        }
        
        \draw[line width=1pt] (0,0) rectangle (9,5);
        \node[anchor=west] at (0,5.35) {$z^{(k)}$};
        
        \node[align=center,scale=0.85] at (4.5,-0.45) {%
            \tikz{\draw[line width=2pt,blue!60] (0,0) rectangle (0.5,0.22);\fill[pattern=north east lines,pattern color=blue!35] (0,0) rectangle (0.5,0.22);}~$F_{N_{k+1}}$
            \quad\quad
            \tikz{\draw[green!60] (0,0) rectangle (0.5,0.22);\fill[green!18] (0,0) rectangle (0.5,0.22);}~good copy
            \quad\quad
            \tikz{\draw[red!60] (0,0) rectangle (0.5,0.22);\fill[pattern=dots,pattern color=red!55] (0,0) rectangle (0.5,0.22);}~bad copy
        };
    \end{tikzpicture}
    
    \vspace{1.0ex}
    
    \begin{tikzpicture}[>=Latex,scale=0.9]
        \def\W{1.125}
        \def\H{1.0}
        
        \foreach \c in {0,...,7}{
            \foreach \r in {0,...,4}{
                \draw[very thin,draw=black!35] ({\c*\W},{\r*\H}) rectangle ({\c*\W+\W},{\r*\H+\H});
            }
        }
        
        \fill[pattern=north east lines,pattern color=blue!35] ({2*\W},{1*\H}) rectangle ({6*\W},{4*\H});
        \draw[line width=2pt,blue!60] ({2*\W},{1*\H}) rectangle ({6*\W},{4*\H});
        \node at ({4*\W},{2.5*\H}) {$F_{N_{k+1}}$};
        
        \foreach \c/\r in {0/0,0/2,0/4,1/1,1/3, 6/0,6/2,6/4,7/1,7/3, 2/0,4/0, 3/4,5/4}{
            \fill[green!18] ({\c*\W},{\r*\H}) rectangle ({\c*\W+\W},{\r*\H+\H});
            \draw[green!60]     ({\c*\W},{\r*\H}) rectangle ({\c*\W+\W},{\r*\H+\H});
        }
        
        \foreach \c/\r in {0/1,0/3,1/0,1/2,1/4, 6/1,6/3,7/0,7/2,7/4, 3/0,5/0, 2/4,4/4}{
            \fill[pattern=horizontal lines,pattern color=red!70] ({\c*\W},{\r*\H}) rectangle ({\c*\W+\W},{\r*\H+\H});
            \draw[red!60]                                  ({\c*\W},{\r*\H}) rectangle ({\c*\W+\W},{\r*\H+\H});
        }
        
        \draw[line width=1pt] (0,0) rectangle (9,5);
        \node[anchor=west] at (0,5.35) {$z^{(k+1)}$};
        
        \node[align=center,scale=0.85] at (4.5,-0.45) {%
            \tikz{\draw[line width=2pt,blue!60] (0,0) rectangle (0.5,0.22);\fill[pattern=north east lines,pattern color=blue!35] (0,0) rectangle (0.5,0.22);}~$F_{N_{k+1}}$
            \quad\quad
            \tikz{\draw[green!60] (0,0) rectangle (0.5,0.22);\fill[green!18] (0,0) rectangle (0.5,0.22);}~good copy
            \quad\quad
            \tikz{\draw[red!60] (0,0) rectangle (0.5,0.22);\fill[pattern=horizontal lines,pattern color=red!70] (0,0) rectangle (0.5,0.22);}~bad copy $\to$ replaced
        };
    \end{tikzpicture}
    
\caption{
One step of the block-replacement construction. The finite region
$W_{N_{k+1}}$ is frozen; outside it, good block types are preserved and all
other block types are replaced by a fixed reference block.}
\label{fig:block-replacement}
\end{figure}

We now perform the main block-replacement construction. The only technical
point is that the decision whether a block is good should depend only on the
block itself, not on its position. For this reason we first separate the
interior of a block from its boundary.

Fix once and for all an increasing sequence $(Q_N)_{N\in\N}$ of finite subsets
of $G$ such that
\[
	\bigcup_{N\in\N} Q_N=G.
\]
For $N\in\N$, put
\[
	W_N=Q_N\cup\bigcup_{r\le N}F_r .
\]
Thus $(W_N)_{N\in\N}$ is an increasing sequence of finite subsets of $G$ whose
union is $G$.

For $z\in\alf^G$, $L\in\N$, and $h\in H_L$, write
\[
	\operatorname{blk}_{L,h}(z)=(z_{ah})_{a\in F_L}\in\alf^{F_L}
\]
for the block of $z$ over the principal copy $F_Lh$.

For a finite set $D\subset G$ and a principal copy $S=F_Lh$, define the
$D$-interior of $S$ by
\[
	\operatorname{Int}_D(S)=\{g\in S: Dg\subset S\}.
\]
If $w\in\alf^*$, write $D_w=\operatorname{dom}(w)$. For a configuration
$c\in\alf^{F_L}$, define the interior frequency of $w$ in $c$ by
\[
	q_{L,w}(c)
	=
	\frac{1}{|F_L|}
	\bigl|
	\{a\in F_L:
	D_w a\subset F_L
	\text{ and } c_{da}=w_d \text{ for all } d\in D_w\}
	\bigr|.
\]
Thus $q_{L,w}(c)$ depends only on the configuration $c$ on $F_L$. Moreover, if
$S=F_Lh$ and $c=\operatorname{blk}_{L,h}(z)$, then
\begin{equation}\label{eq:interior-frequency-bound}
	\bigl|
	\m(z,S)([w])-q_{L,w}(c)
	\bigr|
	\le
	\frac{|S\setminus\operatorname{Int}_{D_w}(S)|}{|S|}.
\end{equation}

We shall use the following two elementary facts.

\begin{lem}\label{lem:finite-modification}
	Let $u,v\in\alf^G$ differ on only finitely many coordinates. Then $u$ and
	$v$ have the same empirical limits along every F{\o}lner sequence. The same
	is true after applying any finite composition of jigsaws. In particular, if
	$u$ is universally good and generic for an ergodic measure, then so is $v$.
\end{lem}

\begin{proof}
	Let
	\[
		P=\{g\in G:u_g\neq v_g\}.
	\]
	Then $P$ is finite. Fix a cylinder $[w]$ and write
	$D_w=\operatorname{dom}(w)$. The values
	$\mathbf 1_{[w]}(gu)$ and $\mathbf 1_{[w]}(gv)$ may differ only when
	\[
		D_wg\cap P\neq\emptyset .
	\]
	The set of such $g$ is finite. Hence it has density zero along any
	F{\o}lner sequence. Therefore $u$ and $v$ have the same empirical limits.

	A jigsaw of fixed level maps configurations which differ on finitely many
	coordinates to configurations which again differ on finitely many
	coordinates, since only finitely many principal blocks are affected. The
	same holds for any finite composition of jigsaws. The claim about universal
	goodness follows directly from its definition.
\end{proof}

\begin{lem}\label{lem:finite-exception-ue}
	Let $z\in\alf^G$. Suppose that for every $w\in\alf^*$ and every
	$\delta>0$ there exist $K\in\N$ and a finite set $E\subset H_K$ such that
	the numbers
	\[
		\bigl\{\m(z,F_Kg)([w]):g\in H_K\setminus E\bigr\}
	\]
	are pairwise $\delta$-close. Then $z$ is uniquely ergodic.
\end{lem}

\begin{proof}
	Fix $w\in\alf^*$ and $\eta>0$. Choose $K\in\N$ and a finite set
	$E\subset H_K$ such that the numbers
	\[
		\bigl\{\m(z,F_Kg)([w]):g\in H_K\setminus E\bigr\}
	\]
	are pairwise $\eta$-close.

	We claim that the conclusion of Lemma~\ref{lemma:1} still holds with this
	finite exceptional set. The proof is the same as the proof of
	Lemma~\ref{lemma:1}, except that one ignores the finitely many principal
	copies
	\[
		F_Kg,\qquad g\in E.
	\]
	Indeed, when a large right translate $F_ng$ is approximated by a union of
	principal copies of $F_K$, at most $|E|$ of these copies are exceptional.
	Their total contribution is bounded by
	\[
		\frac{|E||F_K|}{|F_n|},
	\]
	which tends to $0$ as $n\to\infty$, uniformly in the translate. The usual
	boundary contribution is controlled by the F{\o}lner property, exactly as
	in Lemma~\ref{lemma:1}.

	Consequently, for every $\gamma>0$ there exists $N\in\N$ such that the
	numbers
	\[
		\bigl\{\m(z,F_ng)([w]):g\in G,\ n\ge N\bigr\}
	\]
	are pairwise $(2\eta+\gamma)$-close. Since $\eta>0$ and $\gamma>0$ were
	arbitrary, the hypothesis of Lemma~\ref{lemma:2} is satisfied. Therefore
	$z$ is uniquely ergodic.
\end{proof}

Theorem~\ref{main} is the main result of this section.

\begin{thm}\label{main}
	Let $\F=(F_n)_{n\in\N}$ be a tempered Cortez--Petite F{\o}lner sequence
	satisfying
	\[
		\frac{|F_n|}{|F_{n+1}|}\longrightarrow 0
		\qquad\text{as } n\to\infty.
	\]
	Let $x\in\alf^G$ be a universally good point which is generic for an
	ergodic measure with respect to $\F$. 
    Then $x$ is a $\dbar_{\F}$-limit of a sequence of points whose orbit closures
are uniquely ergodic. Equivalently, each point in the approximating sequence is
generic, with respect to $\F$, for the unique invariant measure on its orbit
closure.
\end{thm}

\begin{proof}
	It is enough to prove that for every $\eps>0$ there exists a point
	$z\in\alf^G$ such that
	\[
		\dbar_{\F}(x,z)\le \eps
	\]
	and the orbit closure of $z$ is uniquely ergodic. Applying this with
	$\eps=1/r$, $r\in\N$, gives the desired approximating sequence.

	Fix $\eps>0$, and enumerate the words in $\alf^*$ as
	\[
		w_1,w_2,\ldots .
	\]
	We construct inductively points $z^{(k)}\in\alf^G$, integers
	$L_k,N_k\in\N$, finite exceptional sets
	\[
		E_k\subset H_{L_k},
	\]
	and collections of good block types
	\[
		\mathcal G_k\subset\alf^{F_{L_k}}.
	\]
	The construction will satisfy the following conditions.

	\begin{enumerate}
		\item\label{main:generic}
		For every $k$, the point $z^{(k)}$ is universally good and is generic
		for some ergodic measure $\mu_k$ with respect to $\F$.

		\item\label{main:increasing}
		$L_k\to\infty$ and $N_k\to\infty$.

		\item\label{main:stabilization}
		\[
			z^{(k+1)}_{|W_{N_k}}=z^{(k)}_{|W_{N_k}}
			\qquad\text{for every } k.
		\]

		\item\label{main:dbar}
		For every $k$,
		\[
			\sup_{n\in\N}
			\frac{1}{|F_n|}
			\#\{g\in F_n:x_g\neq z^{(k)}_g\}
			\le
			\eps(1-2^{-k}).
		\]

		\item\label{main:good-blocks}
		For every $i\le k$ and every $h\in H_{L_i}\setminus E_i$, one has
		\[
			\operatorname{blk}_{L_i,h}(z^{(k)})\in\mathcal G_i .
		\]
	\end{enumerate}

	We start with
	\[
		z^{(1)}=x,
		\qquad
		L_1=N_1=1,
		\qquad
		E_1=\emptyset,
		\qquad
		\mathcal G_1=\alf^{F_{L_1}}.
	\]
	Then all inductive conditions are immediate.

	Suppose that the construction has been carried out up to stage $k$. Let
	$\mu_k$ be the ergodic measure for which $z^{(k)}$ is generic. Put
	\[
		\alpha_{k+1}=\frac{\eps}{2^{k+3}},
		\qquad
		\rho_{k+1}=\frac{\eps}{20(k+1)},
		\qquad
		\beta_{k+1}=\min\{\alpha_{k+1},\rho_{k+1}\}.
	\]
	Let
	\[
		D_{k+1}=\bigcup_{j=1}^{k+1}D_{w_j}.
	\]

	Choose $L=L_{k+1}>N_k$ so large that
	\begin{equation}\label{eq:main-boundary-small}
		\frac{|F_L\setminus\operatorname{Int}_{D_{k+1}}(F_L)|}{|F_L|}
		<\rho_{k+1},
	\end{equation}
	and such that the conclusion of Lemma~\ref{ostatnilemat}, applied to the
	point $z^{(k)}$ and the words $w_1,\ldots,w_{k+1}$, holds at level $L$
	with tolerance $\beta_{k+1}$. Thus, for all sufficiently large $n$, there
	are at least
	\[
		(1-\alpha_{k+1})\frac{|F_n|}{|F_L|}
	\]
	principal copies $S$ of $F_L$ inside $F_n$ such that
	\begin{equation}\label{eq:main-good-full}
		\bigl|
		\m(z^{(k)},S)([w_j])-\mu_k([w_j])
		\bigr|
		<
		\rho_{k+1}
		\qquad\text{for } j=1,\dots,k+1.
	\end{equation}

	Let
	\[
		B_k
		=
		W_{N_k}
		\cup
		\bigcup_{i\le k}\bigcup_{h\in E_i}F_{L_i}h .
	\]
	This is a finite set. Choose a principal copy
	\[
		R=F_Lh_*
	\]
	which satisfies~\eqref{eq:main-good-full} and is disjoint from $B_k$. Such
	a copy exists because only finitely many principal copies of $F_L$ intersect
	$B_k$, whereas good copies occur with positive asymptotic density.

	Let
	\[
		c_*=\operatorname{blk}_{L,h_*}(z^{(k)})\in\alf^{F_L}
	\]
	be the block carried by this reference copy. Define
	\[
		\mathcal G_{k+1}
		=
		\left\{
		c\in\alf^{F_L}:
		\bigl|
		q_{L,w_j}(c)-\mu_k([w_j])
		\bigr|
		<2\rho_{k+1}
		\text{ for } j=1,\dots,k+1
		\right\}.
	\]
	By~\eqref{eq:interior-frequency-bound},
	\eqref{eq:main-boundary-small}, and~\eqref{eq:main-good-full}, the
	reference block $c_*$ belongs to $\mathcal G_{k+1}$.

	Define a block map
	\[
		\psi_{k+1}\colon\alf^{F_L}\to\alf^{F_L}
	\]
	by
	\[
		\psi_{k+1}(c)
		=
		\begin{cases}
			c, & \text{if } c\in\mathcal G_{k+1},\\
			c_*, & \text{if } c\notin\mathcal G_{k+1}.
		\end{cases}
	\]
	Let $\Psi_{k+1}$ be the corresponding jigsaw of level $L$, and set
	\[
		\widetilde z^{(k+1)}=\Psi_{k+1}(z^{(k)}).
	\]
	By Lemma~\ref{nietracimy}, the point $\widetilde z^{(k+1)}$ is universally
	good and generic for an ergodic measure.

	Choose $N=N_{k+1}>L$ so large that $N>N_k$ and, for every $n\ge N$:
	\begin{enumerate}
		\item the principal copies of $F_L$ which intersect $F_n$ but are not
		contained in $F_n$ cover at most $\alpha_{k+1}|F_n|$ points;
		\item the principal copies of $F_L$ contained in $F_n$ whose block type is
not in $\mathcal G_{k+1}$ cover at most $\alpha_{k+1}|F_n|$ points.
	\end{enumerate}
	The first condition follows from the F{\o}lner property, and the second
	from~\eqref{eq:main-good-full} together with
	\eqref{eq:interior-frequency-bound}.

	Let
	\[
		E_{k+1}
		=
		\{h\in H_L:F_Lh\cap W_N\neq\emptyset\}.
	\]
	This is finite. Define $z^{(k+1)}$ blockwise on principal copies
	$F_Lh$, $h\in H_L$, by
	\[
		z^{(k+1)}_{|F_Lh}
		=
		\begin{cases}
			z^{(k)}_{|F_Lh}, & \text{if } h\in E_{k+1},\\
			\widetilde z^{(k+1)}_{|F_Lh}, & \text{if } h\notin E_{k+1}.
		\end{cases}
	\]
	Then
	\[
		z^{(k+1)}_{|W_N}=z^{(k)}_{|W_N}.
	\]
	Moreover, $z^{(k+1)}$ differs from $\widetilde z^{(k+1)}$ only on finitely
	many coordinates. Hence, by Lemma~\ref{lem:finite-modification},
	$z^{(k+1)}$ is universally good and generic for an ergodic measure. This
	proves~\ref{main:generic} at stage $k+1$.

	Conditions~\ref{main:increasing} and~\ref{main:stabilization} are immediate
	from the choice of $L$ and $N$. We now prove~\ref{main:dbar}. If $n\le N$,
	then $F_n\subset W_N$, and no point of $F_n$ is changed. If $n>N$, changes
	inside $F_n$ may occur only in principal copies of $F_L$ whose block type is
	not in $\mathcal G_{k+1}$, together with boundary copies which intersect
	$F_n$ without being contained in $F_n$. By the choice of $N$, these two
	parts have total cardinality at most
	\[
		2\alpha_{k+1}|F_n|
		\le
		\frac{\eps}{2^{k+1}}|F_n|.
	\]
	Therefore
	\[
		\sup_{n\in\N}
		\frac{1}{|F_n|}
		\#\{g\in F_n:z^{(k)}_g\neq z^{(k+1)}_g\}
		\le
		\frac{\eps}{2^{k+1}}.
	\]
	Combining this with the inductive estimate for $z^{(k)}$ gives
	\[
		\sup_{n\in\N}
		\frac{1}{|F_n|}
		\#\{g\in F_n:x_g\neq z^{(k+1)}_g\}
		\le
		\eps(1-2^{-(k+1)}).
	\]

	It remains to verify~\ref{main:good-blocks}. For $i=k+1$, this follows
	directly from the definition of $\psi_{k+1}$ and $E_{k+1}$. Now fix
	$i\le k$ and $h\in H_{L_i}\setminus E_i$. The principal copy $F_{L_i}h$ is
	contained in a unique principal copy $F_Lt$ of $F_L$. If $t\in E_{k+1}$, or
	if $\operatorname{blk}_{L,t}(z^{(k)})\in\mathcal G_{k+1}$, then the block
	over $F_{L_i}h$ is unchanged, and the claim follows from the inductive
	hypothesis.

	Otherwise, the copy $F_Lt$ is replaced by the reference block $c_*$. Hence
	the new block over $F_{L_i}h$ coincides with a block of $z^{(k)}$ over some
	principal copy of $F_{L_i}$ contained in $R$. Since $R\cap B_k=\emptyset$,
	this principal copy is not one of the exceptional copies indexed by $E_i$.
	Therefore, by the inductive hypothesis, that block belongs to
	$\mathcal G_i$. This proves~\ref{main:good-blocks} at stage $k+1$ and
	completes the induction.

	By~\ref{main:increasing} and~\ref{main:stabilization}, the sequence
	$(z^{(k)})_{k\in\N}$ converges in the product topology. Indeed, every
	finite subset of $G$ is eventually contained in some $W_{N_k}$. Let
	\[
		z=\lim_{k\to\infty}z^{(k)}.
	\]

	We now show that $\dbar_{\F}(x,z)\le\eps$. Fix $n\in\N$. Since
	$F_n$ is finite and the sets $W_{N_k}$ exhaust $G$, there exists
	$K\in\N$ such that
	\[
		F_n\subset W_{N_K}.
	\]
	By the stabilization property, the configuration on $F_n$ is no longer
	changed after stage $K$. Hence
	\[
		z_{|F_n}=z^{(K)}_{|F_n}.
	\]
	Therefore, using~\ref{main:dbar},
	\[
		\frac{1}{|F_n|}
		\#\{g\in F_n:x_g\neq z_g\}
		=
		\frac{1}{|F_n|}
		\#\{g\in F_n:x_g\neq z^{(K)}_g\}
		\le
		\eps(1-2^{-K})
		\le
		\eps.
	\]
	Since $n\in\N$ was arbitrary, we get
	\[
		\dbar_{\F}(x,z)\le\eps.
	\]

	It remains to prove that $z$ is uniquely ergodic. Fix $w_j\in\alf^*$ and
	$\delta>0$. Choose $i>j$ so large that
	\[
		6\rho_i<\delta.
	\]
	Let $h\in H_{L_i}\setminus E_i$. Since $F_{L_i}h$ is finite, there exists
	$K\ge i$ such that
	\[
		F_{L_i}h\subset W_{N_K}.
	\]
	By stabilization,
	\[
		z_{|F_{L_i}h}=z^{(K)}_{|F_{L_i}h}.
	\]
	By~\ref{main:good-blocks},
	\[
		\operatorname{blk}_{L_i,h}(z)\in\mathcal G_i.
	\]
	Hence
	\[
		\bigl|
		q_{L_i,w_j}(\operatorname{blk}_{L_i,h}(z))
		-\mu_{i-1}([w_j])
		\bigr|
		<2\rho_i.
	\]
	Using~\eqref{eq:interior-frequency-bound} and the choice of $L_i$, we get
	\[
		\bigl|
		\m(z,F_{L_i}h)([w_j])-\mu_{i-1}([w_j])
		\bigr|
		<3\rho_i.
	\]
	Therefore the numbers
	\[
		\bigl\{
		\m(z,F_{L_i}h)([w_j]):
		h\in H_{L_i}\setminus E_i
		\bigr\}
	\]
	are pairwise $6\rho_i$-close, hence pairwise $\delta$-close.

	Since $w_j$ and $\delta$ were arbitrary, Lemma~\ref{lem:finite-exception-ue}
	implies that $z$ is uniquely ergodic. Finally, every weak$^*$ accumulation
	point of the empirical measures of $z$ along $\F$ is an invariant measure
	supported on $\overline{O_G(z)}$. Since this orbit closure is uniquely
	ergodic, these empirical measures converge to its unique invariant measure.
	Thus $z$ is generic for the unique invariant measure on its orbit closure.
This completes the proof.
\end{proof}

Using Theorems~\ref{HD} and~\ref{unequ}, together with
Corollary~\ref{cor:entropy-continuity}, we obtain the following corollary of
Theorem~\ref{main}.
\begin{cor}\label{glowny}
	Let $G$ be an infinite countable amenable residually finite group, and let
	$\F=(F_n)_{n\in\N}$ be a tempered Cortez--Petite F{\o}lner sequence in $G$
	such that
	\[
		\frac{|F_n|}{|F_{n+1}|}\longrightarrow 0
		\qquad\text{as } n\to\infty.
	\]
	Then the invariant measures supported on uniquely ergodic subsystems are
	entropy dense in $\M_G^e(\alf^G)$.
\end{cor}

\begin{proof}
	Let $\mu\in\M_G^e(\alf^G)$. By Lemma~\ref{lem:universal}, we may choose a
	universally good point $x\in\alf^G$ which is generic for $\mu$ with respect
	to $\F$.
	By Theorem~\ref{main}, for every $r\in\N$ there exists a point
	$x_r\in\alf^G$ such that
	\[
		\dbar_{\F}(x_r,x)<\frac1r
	\]
	and $x_r$ is generic, with respect to $\F$, for a measure $\mu_r$ whose
	orbit closure is uniquely ergodic. In particular,
	\[
		\M_G(\overline{O_G(x_r)})=\{\mu_r\}.
	\]
	By Theorem~\ref{unequ}, the convergence
	\[
		\dbar_{\F}(x_r,x)\longrightarrow 0
	\]
	implies
	\[
		D_{B,\F}(x_r,x)\longrightarrow 0.
	\]
	Hence, by Theorem~\ref{HD},
	\[
		\HD\bigl(\hat\omega_{\F}(x_r),\hat\omega_{\F}(x)\bigr)
		\longrightarrow 0.
	\]
	Since $x$ is $\F$-generic for $\mu$ and $x_r$ is $\F$-generic for $\mu_r$,
	we have
	\[
		\hat\omega_{\F}(x)=\{\mu\},
		\qquad
		\hat\omega_{\F}(x_r)=\{\mu_r\}.
	\]
	Therefore
	\[
		\mu_r\xrightarrow[r\to\infty]{w^*}\mu.
	\]
Moreover, by Corollary~\ref{cor:entropy-continuity}, the convergence
	\[
		\dbar_{\F}(x_r,x)\longrightarrow 0
	\]
	implies
	\[
		h(\mu_r)\longrightarrow h(\mu).
	\]
    Thus $\mu$ is the weak$^*$ and entropy limit of invariant measures
	supported on uniquely ergodic subsystems. Since
	$\mu\in\M_G^e(\alf^G)$ was arbitrary, the claim follows.
\end{proof}

\appendix

\section{Continuity of entropy}\label{appendix}
We assume $|\alf|\ge 2$; the case $|\alf|=1$ is trivial.
We shall use the following elementary facts about static entropy. The proofs
can be found in standard texts on information theory, for instance
\cite{gray2011entropy}.

If $\Pe$ and $\Qe$ are finite measurable partitions, the conditional entropy
of $\Pe$ with respect to $\Qe$ is
\[
	H_\lambda(\Pe\mid\Qe)
	=
	H_\lambda(\Pe\vee\Qe)-H_\lambda(\Qe).
\]
The chain rule for entropy states that
\[
	H_\lambda(\Pe_1\vee\Pe_2\vee\cdots\vee\Pe_n)
	=
	\sum_{i=1}^n
	H_\lambda(\Pe_i\mid \Pe_1\vee\cdots\vee\Pe_{i-1}),
\]
with the convention that the first term is $H_\lambda(\Pe_1)$.

Set
\[
	h_2(t)=-t\log t-(1-t)\log(1-t),
\]
with the convention $h_2(0)=h_2(1)=0$.

\begin{lem}[Fano inequality, Lemma 4.2.1 in \cite{gray2011entropy}]
	Let
	\[
		\Pe=\{P_a:a\in\alf\}
		\qquad\text{and}\qquad
		\Qe=\{Q_a:a\in\alf\}
	\]
	be two finite partitions indexed by the same finite alphabet $\alf$. Define
	the error probability by
	\[
		P_e
		=
		\frac12\sum_{a\in\alf}\lambda(P_a\triangle Q_a).
	\]
	Then
	\[
		H_\lambda(\Pe\mid\Qe)
		\le
		h_2(P_e)+P_e\log(|\alf|-1).
	\]
	The same bound holds for $H_\lambda(\Qe\mid\Pe)$.
\end{lem}

For invariant measures $\mu,\nu\in\M_G(\alf^G)$, let $\mathcal J(\mu,\nu)$
denote the set of invariant joinings of $\mu$ and $\nu$. We use the joining
description of the $\dbar$-distance:
\[
	\dbar_{\F}(\mu,\nu)
	=
	\inf_{\lambda\in\mathcal J(\mu,\nu)}
	\lambda\{(x,y)\in\alf^G\times\alf^G:x_e\neq y_e\}.
\]
The infimum is a minimum, since $\mathcal J(\mu,\nu)$ is weak$^*$ compact and
the function
\[
	\lambda\mapsto
	\lambda\{(x,y):x_e\neq y_e\}
\]
is weak$^*$ continuous.

\begin{lem}\label{lem:infmin}
	Let $\mu,\nu\in\M_G^e(\alf^G)$. Then there exists an ergodic joining
	$\lambda\in\mathcal J(\mu,\nu)$ such that
	\[
		\dbar_{\F}(\mu,\nu)
		=
		\lambda\{(x,y)\in\alf^G\times\alf^G:x_e\neq y_e\}.
	\]
\end{lem}

\begin{proof}
	Let $\lambda\in\mathcal J(\mu,\nu)$ be a joining attaining the minimum.
	Write its ergodic decomposition as
	\[
		\lambda=\int \lambda_\omega\,d\theta(\omega).
	\]
	For $\theta$-almost every $\omega$, the measure $\lambda_\omega$ is
	$G$-invariant and ergodic on $\alf^G\times\alf^G$. Its first and second
	marginals are $G$-invariant measures whose averages are $\mu$ and $\nu$,
	respectively. Since $\mu$ and $\nu$ are ergodic, they are extreme points of
	$\M_G(\alf^G)$. Hence, for $\theta$-almost every $\omega$, the first
	marginal of $\lambda_\omega$ is $\mu$ and the second marginal is $\nu$.
	Thus $\lambda_\omega\in\mathcal J(\mu,\nu)$ for $\theta$-almost every
	$\omega$.

	Put
	\[
		\Delta_e=\{(x,y)\in\alf^G\times\alf^G:x_e\neq y_e\}.
	\]
	Since $\lambda$ attains the minimum,
	\[
		\dbar_{\F}(\mu,\nu)
		=
		\lambda(\Delta_e)
		=
		\int \lambda_\omega(\Delta_e)\,d\theta(\omega).
	\]
	On the other hand, for $\theta$-almost every $\omega$,
	\[
		\lambda_\omega(\Delta_e)\ge \dbar_{\F}(\mu,\nu),
	\]
	because $\lambda_\omega$ is also a joining of $\mu$ and $\nu$. Therefore
	\[
		\lambda_\omega(\Delta_e)=\dbar_{\F}(\mu,\nu)
	\]
	for $\theta$-almost every $\omega$. Choosing such an $\omega$ gives the
	required ergodic joining.
\end{proof}

\kalikow*

\begin{proof}
	Let $\mathcal S$ be the standard partition of $\alf^G$ according to the
	symbol at the identity:
	\[
		\mathcal S=\{[a]_e:a\in\alf\},
		\qquad
		[a]_e=\{x\in\alf^G:x_e=a\}.
	\]
Since $\mathcal S$ is a generating partition for the full shift, we have
	\[
		h(\mu)=h_\mu(\mathcal S,G)
		\qquad\text{and}\qquad
		h(\nu)=h_\nu(\mathcal S,G).
	\]
    
	Put
	\[
		\Delta_e=\{(x,y)\in\alf^G\times\alf^G:x_e\neq y_e\}.
	\]
	By Lemma~\ref{lem:infmin}, choose an ergodic joining
	$\lambda\in\mathcal J(\mu,\nu)$ such that
	\[
		\lambda(\Delta_e)=\dbar_{\F}(\mu,\nu).
	\]

	Let
	\[
		\mathcal P=\{[a]_e\times\alf^G:a\in\alf\}
	\]
	be the partition according to the symbol at $e$ in the first coordinate,
	and let
	\[
		\mathcal Q=\{\alf^G\times[a]_e:a\in\alf\}
	\]
	be the corresponding partition for the second coordinate. Then
	\[
		\frac12\sum_{a\in\alf}
		\lambda\bigl(([a]_e\times\alf^G)
		\triangle
		(\alf^G\times[a]_e)\bigr)
		=
		\lambda(\Delta_e)
		=
		\dbar_{\F}(\mu,\nu).
	\]

	Fix a finite set $F\subset G$. We have
	\[
		H_\lambda(\mathcal P^F)=H_\mu(\mathcal S^F),
		\qquad
		H_\lambda(\mathcal Q^F)=H_\nu(\mathcal S^F).
	\]
	Therefore
	\[
		H_\mu(\mathcal S^F)-H_\nu(\mathcal S^F)
		=
		H_\lambda(\mathcal P^F)-H_\lambda(\mathcal Q^F)
		\le
		H_\lambda(\mathcal P^F\mid\mathcal Q^F).
	\]
	Similarly,
	\[
		H_\nu(\mathcal S^F)-H_\mu(\mathcal S^F)
		\le
		H_\lambda(\mathcal Q^F\mid\mathcal P^F).
	\]

	Enumerate $F=\{f_1,\ldots,f_n\}$. By the chain rule and monotonicity of
	conditional entropy under conditioning, we get
	\[
		H_\lambda(\mathcal P^F\mid\mathcal Q^F)
		\le
		\sum_{i=1}^n
		H_\lambda(\mathcal P^{-f_i}\mid\mathcal Q^{-f_i}).
	\]
	For each $i$, the error probability between the labelled partitions
	$\mathcal P^{-f_i}$ and $\mathcal Q^{-f_i}$ is equal to
	$\lambda(\Delta_e)$, by $G$-invariance of $\lambda$. Hence Fano's
	inequality gives
	\[
		H_\lambda(\mathcal P^{-f_i}\mid\mathcal Q^{-f_i})
		\le
		h_2(\dbar_{\F}(\mu,\nu))
		+
		\dbar_{\F}(\mu,\nu)\log(|\alf|-1).
	\]
	Thus
	\[
		H_\lambda(\mathcal P^F\mid\mathcal Q^F)
		\le
		|F|
		\left(
		h_2(\dbar_{\F}(\mu,\nu))
		+
		\dbar_{\F}(\mu,\nu)\log(|\alf|-1)
		\right).
	\]
	The same argument, with $\mathcal P$ and $\mathcal Q$ interchanged, gives
	the same bound for
	$H_\lambda(\mathcal Q^F\mid\mathcal P^F)$. Consequently,
	\[
		\left|
		\frac{1}{|F|}H_\mu(\mathcal S^F)
		-
		\frac{1}{|F|}H_\nu(\mathcal S^F)
		\right|
		\le
		h_2(\dbar_{\F}(\mu,\nu))
		+
		\dbar_{\F}(\mu,\nu)\log(|\alf|-1).
	\]

	Applying this to $F=F_n$ and passing to the limit, we obtain
	\[
		|h(\mu)-h(\nu)|
		\le
		h_2(\dbar_{\F}(\mu,\nu))
		+
		\dbar_{\F}(\mu,\nu)\log(|\alf|-1).
	\]
	This completes the proof.
\end{proof}

\section*{Acknowledgements}
The authors are grateful to Dominik Kwietniak for many valuable suggestions and remarks. MM was supported by the National Science Centre (NCN), Poland, grant no.~2022/47/B/ST1/02866.

\section*{Use of AI-assisted tools}

AI-assisted tools were used to support language editing, proofreading, and
improving the clarity of exposition. AI-assisted tools were also used in
discussions of the presentation of selected proofs. The authors reviewed and
verified all edits and take full responsibility for the mathematical content and
the final version of the manuscript.
  \bibliographystyle{acm}
\def\bibname{Literature}

\bibliography{bibliography}

\end{document}